\documentclass[11pt]{amsart}
\usepackage{amsmath}
\usepackage{amssymb}
\usepackage{amscd}

\def\NZQ{\mathbb}               
\def\NN{{\NZQ N}}
\def\QQ{{\NZQ Q}}
\def\ZZ{{\NZQ Z}}
\def\RR{{\NZQ R}}

\newtheorem{Theorem}{Theorem}[section]
\newtheorem{Lemma}[Theorem]{Lemma}

\newtheorem{Proposition}[Theorem]{Proposition}
\newtheorem{Remark}[Theorem]{Remark}

\newtheorem{Definition}[Theorem]{Definition}

\newtheorem{Question}[Theorem]{Question}
\let\epsilon\varepsilon
\let\phi=\varphi
\let\kappa=\varkappa

\begin{document}
\title{Local Uniformization of Abhyankar valuations}
\author{Steven Dale Cutkosky}
\thanks{The  author was partially supported by NSF grant DMS-1700046.}

\address{Steven Dale Cutkosky, Department of Mathematics,
University of Missouri, Columbia, MO 65211, USA}
\email{cutkoskys@missouri.edu}

\begin{abstract} We prove local uniformization of Abhyankar valuations of an algebraic function field $K$ over a ground field $k$. Our result generalizes the proof of this result, with the additional assumption that the residue field of the valuation ring is separable over $k$,  by Hagen Knaf and Franz-Viktor Kuhlmann. The proof in this paper uses different methods, being inspired by the approach of Zariski and Abhyankar. 
\end{abstract}


\maketitle

\section{Introduction}\label{SecInt} 

In this paper we prove local uniformization of Abhyankar valuations $\nu$ of an algebraic function field $K$ over a ground field $k$. An Abhyankar valuation is a  valuation which satisfies equality in Abhyankar's inequality (\ref{eq27}). These valuations are particularly well behaved. Abhyankar \cite{Ab1} showed that  the value groups of these valuations are finitely generated, and that the residue fields  of their valuation rings are finitely 
generated field extensions of $k$.  In \cite[Theorem 1.1]{KK}, Knaf and Kuhlmann prove that with the additional assumption that the residue field of the valuation is separable over the ground field $k$, local uniformization holds for Abhyankar valuations of algebraic function fields.  A version of this theorem, valid for Abhyankar valuations in complete local rings over an algebraically closed field, is proven by Teissier in \cite[Theorem 5.5.1]{T}. In this paper, we prove that local uniformization holds for Abhyankar valuations in algebraic function fields, without any extra assumptions. Our local uniformization theorems are given in Theorems \ref{TheoremA}, \ref{TheoremB} and \ref{TheoremC}, stated later in this introduction, and proven in this paper. 

The proof of Knaf and Kuhlmann \cite{KK}, which has the assumption that the residue field of the Abhyankar valuation $\nu$ is separable over $k$,  shows that there is a regular local ring $R$ of $K$ which is dominated by the valuation $\nu$ such that $R$ is smooth over the ground field $k$. Without the assumption that the residue field of $\nu$ is separable over $k$, this may not be possible to achieve. However, we  prove in the general case of an Abhyankar valuation, that there exists a regular local ring $R$ of $K$ which is dominated by the valuation $\nu$.

Our proof is a generalization of the proof of Zariski for maximal rational rank valuations in a characteristic zero algebraic function field, \cite{Z1}. This method was used by Abhyankar to prove local uniformization of Abhyankar valuations in two dimensional algebraic function fields over an algebraically closed ground field in \cite[Section 1]{Ab2}. 
The proofs in \cite{Z1} and \cite{Ab2} both make use of the values of derivations of $K/k$  to achieve reduction of multiplicity. We only use the definition of a regular local ring: it has a regular system of parameters. 
Zariski used Perron transforms in \cite{Z1} to prove local uniformization for rank 1 valuations in characteristic zero algebraic function fields, and made a reduction argument to use this result to prove local uniformization of arbitrary rank valuations in characteristic zero algebraic function fields. 
Our  approach is influenced by that of Samar El Hitti in \cite{E1}, where local uniformization is proven in characteristic zero algebraic function fields for an arbitrary valuation, via a uniform use of higher rank Perron transforms.

 A delicate point in the construction of a proof of local uniformization of an Abhyankar valuation in the general case, when the residue field of the valuation is not
separable over the ground field $k$, is that it may not be possible to find a coefficient field of the completion of a given local ring dominated by the valuation which contains  $k$.

Before stating our local uniformization theorems, we give some necessary background about valuations.
We refer to \cite{ZS2}, \cite{End} and \cite{EP} for basic facts about valuations. 
Let $K$ be an algebraic function field over a field $k$, and $\nu$ be a valuation of $K/k$; that is, a valuation of $K$ which is trivial on $k$. Let $V_{\nu}$ be the valuation ring of $\nu$ with maximal ideal $m_{\nu}$ and $\Gamma_{\nu}$ be its valuation group. Let $t$ be the rank of $\nu$, and
$$
(0)=P_{t+1}^{\nu}\subset P_{t}^{\nu}\subset \cdots\subset P_2^{\nu}\subset P_1^{\nu}=m_{\nu}
$$
be the chain of prime ideals in $V_{\nu}$. Let
$$
0=\Gamma_0\subset\Gamma_1\subset \cdots\subset \Gamma_t=\Gamma_{\nu}
$$
be the chain of isolated subgroups. For $1\le i\le t$, let $\nu_i$ be the specialization of $\nu$ with valuation ring
$V_{\nu_i}=V_{P^{\nu}_{i}}$.  The maximal ideal $m_{\nu_i}$ of $V_{\nu_i}$ is $m_{\nu_i}=P_i^{\nu}V_{\nu_i}$; in particular,  $\nu_1=\nu$. The value group of $\nu_i$ is $\Gamma_{\nu_i}=\Gamma_{\nu}/\Gamma_{i-1}$. 

Abhyankar's inequality (\cite{Ab1} and \cite[Proposition 2, Appendix 2, page 331]{ZS2}) is  
\begin{equation}\label{eq27}
\mbox{rrank  }\nu +\mbox{trdeg}_kV_{\nu}/m_{\nu}\le \mbox{trdeg}_kK
\end{equation}
where $\mbox{rrank }\nu$ is the rational rank of $\nu$.  When equality holds in (\ref{eq27}), we have that 
$\Gamma_{\nu}\cong\ZZ^n$ as a group for some $n$ and $V_{\nu}/m_{\nu}$ is a finitely generated field over $k$.
This is proven in \cite{Ab1}, and \cite[Proposition 3, page 335, Appendix 2]{ZS2}. Valuations that satisfy equality in (\ref{eq27}) are called Abhyankar valuations. 

 The following three theorems, establishing local uniformization along an Abyhankar valuation in an algebraic function field,  are proven in Section \ref{SecLocUnif} of this paper, as a consequence of the theory developed in Section \ref{SecAbh}. Any notation used in the statements of our local uniformization theorems, which is not defined above,  can be found in Section \ref{SecPrem}.   If $V_{\nu}/m_{\nu}$ is separable over $k$, these theorems are a consequence of \cite[Theorem 1.1]{KK}
\begin{Theorem}\label{TheoremA} Suppose that $K$ is an algebraic function field over a field $k$ and $\nu$ is an Abhyankar valuation of $K/k$. Then there exists a regular algebraic local ring $R$ of $K$ such that $\nu$ dominates $R$.
Further,
\begin{enumerate}
\item[1)] $R$ has a  regular system of parameters $x_{1,1},\ldots,x_{t,r_t}$ such that 
$\nu_j(x_{j,1}),\ldots,\nu_j(x_{j,r_j})$ is a $\ZZ$-basis of $\Gamma_j/\Gamma_{j-1}$ and $P_j(R)=P_j^{\nu}\cap R$ is the regular prime ideal $(x_{j,1},\ldots,x_{t,r_t})$ for $1\le j\le t$.
\item[2)] We have that
$$
\left(R/P_i(R)\right)_{P_i(R)}\cong \left(V/P_i^{\nu}\right)_{P^{\nu}_i}\cong V_{\nu_i}/m_{\nu_i}
$$
for $1\le i\le t$.
\end{enumerate}
\end{Theorem}

Regular parameters as in 1) are called very good parameters of $R$ (Definition \ref{Def1}).

Primitive transforms are defined in Definition \ref{Def2}. They are a particularly simple type of birational transform of a regular local ring. 

 \begin{Theorem}\label{TheoremB} Suppose that $R$ satisfies the conclusions of Theorem \ref{TheoremA}.
\begin{enumerate}
\item[1)] Suppose that $0\ne f\in R$. Then there exists a sequence of primitive transforms (\ref{eq25}) along $\nu$, $R\rightarrow R(1)$, such that $R(1)$ with the resulting very good parameters
$x_{1,1}(1),\ldots,x_{t,r_t}(1)$, satisfies the conclusions of Theorem \ref{TheoremA}, and
$$
f=x_{1,1}(1)^{a_{1,1}}\cdots x_{t,r_t}(1)^{a_{t,r_t}}u
$$
where $a_{1,1},\ldots,a_{t,r_t}\in \NN$ and $u\in R(1)$ is a unit.
\item[2)] Suppose that $I\subset R$ is an ideal. Then there exists a sequence of primitive transforms (\ref{eq25})  along $\nu$, $R\rightarrow R(1)$,  and $a_{1,1},\ldots,a_{t,r_t}\in \NN$ such that 
$$
IR(1)=x_{1,1}(1)^{a_{1,1}}\cdots x_{t,r_t}(1)^{a_{t,r_t}}R(1).
$$
\item[3)] Suppose that $0\ne f\in V_{\nu}$. Then there exits a sequence of primitive transforms (\ref{eq25})
along $\nu$, $R\rightarrow R(1)$, such that 
$$
f=x_{1,1}(1)^{a_{1,1}}\cdots x_{t,r_t}(1)^{a_{t,r_t}}u
$$
where $a_{1,1},\ldots,a_{t,r_t}\in \NN$ and $u\in R(1)$ is a unit.
\end{enumerate}
\end{Theorem}

\begin{Theorem}\label{TheoremC} Suppose that $K$ is an algebraic function field over a field $k$ and $\nu$ is an Abhyankar valuation of $K$. Suppose that $S$ is an algebraic local ring of $K$ which is dominated by $\nu$. Then there exists a birational extension $S\rightarrow R$ such that $R$ is a regular algebraic local ring of $K$ which is dominated by $\nu$ and satisfies the conclusions 1) and 2) of Theorem \ref{TheoremA}.
\end{Theorem}

\subsection{Defect of extensions of valuations} It has become apparent that the possibility of defect in a finite extension of valued fields is the essential obstruction to local uniformization in positive characteristic (\cite{K1}, \cite{Sat} and \cite{CM}). This is somewhat surprising, since defect does not appear explicitly   in the proofs that do exist of local uniformization of arbitrary valuations in a positive characteristic algebraic function field of dimension $\le 3$, including  \cite{Ab2}, \cite{Lip}, \cite{H}, \cite{CJS}, \cite{Ab3}, \cite{C4}, \cite{CoPi}. No general results of local uniformization of arbitrary valuations exist, at the time of this writing,  in positive characteristic algebraic function fields of dimension larger than 3.

 We now define the classical ramification and inertia indices and the defect of a finite extension of valued fields. 

Suppose that $K$ is a  field and $\nu$ is a valuation of $K$. Let  $V_{\nu}$ be the valuation ring of $\nu$ with maximal ideal $m_{\nu}$  and $\Gamma_{\nu}$ be the value group of $\nu$. Suppose that $K\rightarrow L$ is a finite field extension and $\omega$ is an extension of $\nu$ to $L$. We have associated ramification and inertia indices of the extension $\omega$ over $\nu$,
$$
e(\omega/\nu)=[\Gamma_{\omega}:\Gamma_{\nu}]\mbox{ and }
f(\omega/\nu)=[V_{\omega}/m_{\omega}:V_{\nu}/m_{\nu}].
$$
The defect of the extension of $\omega$ over $\nu$ is 
$$
\delta(\omega/\nu)=\frac{[L^h:K^h]}{e(\omega/\nu)f(\omega/\nu)}
$$
where $K^h$ and $L^h$ are  henselizations of the valued fields $K$ and $L$.
This is a positive integer (as shown in \cite{EP}) which is 1 if $V_{\nu}/m_{\nu}$ has characteristic zero and is a power of $p$ if $V_{\nu}/m_{\nu}$ has positive characteristic $p$.

The following theorem is a consequence of \cite[Theorem 1]{K}.

\begin{Theorem}\label{TheoremK} Let $K/k$ be an algebraic function field and $\nu$ be an Abhankar valuation of $K/k$. Suppose that $L$ is a finite extension field of $K$ and $\omega$ is an extension of $\nu$ to $L$. Then the defect $\delta(\omega/\nu)=1$.
\end{Theorem}

This theorem plays an essential role in the proof of the local uniformization theorem of \cite{KK}.

It is shown in \cite{CM}, that Zariski's local uniformization algorithm, which takes place in a finite  extension of arbitrary valued fields,   converges if and only if  there is no defect. In particular,  if a projection to a  regular local ring is chosen in which defect occurs, then the resolution algorithm which we use will not converge.

In our proof of local uniformization (Theorems \ref{TheoremA} - \ref{TheoremC}) the fact that there is no defect in an extension of Abhyankar valuations does not appear explicitly, and we do not use Theorem \ref{TheoremK}. However, since we show explicitly that Zariski's
local uniformization algorithm converges in the completion of our local ring, it is implicit in the proof that there is no defect in our finite extension. 

 
\subsection{Essentially finitely generated extensions of valuation rings}

In this subsection, we discuss a very interesting question proposed by Hagen Knaf, and give an application of our local uniformization theorem to improve a positive result on this question from \cite{CN}.

Let $H$ be an ordered subgroup of an ordered abelian group $G$. The initial index $\epsilon(G/H)$ of $H$ in $G$ is defined (\cite[page 138]{End}) as 
$$
\epsilon(G/H)=|\{g\in G_{\ge 0}\mid g<H_{>0}\}|,
$$
where 
$$
G_{\ge 0}=\{g\in G\mid g\ge 0\}\mbox{ and }H_{>0}=\{h\in H\mid h>0\}.
$$
We define the initial index $\epsilon(\omega/\nu)$ of the finite extension $K\rightarrow L$ as $\epsilon(\Gamma_{\omega}/\Gamma_{\nu})$. 

We always have that $\epsilon(\omega/\nu)\le e(\omega/\nu)$ (\cite[(18.3)]{End}).


Let $D(\nu,L)$ be the integral closure of $V_{\nu}$ in $L$. The localizations of $D(\nu,L)$ at its maximal ideals are the valuation rings $V_{\omega_i}$ of the extensions $\omega_i$ of $\nu$ to $L$. We have the following remarkable theorem.

\begin{Theorem}(\cite[Theorem 18.6]{End}\label{ThmEnd}) The ring $D(\nu,L)$ is a finite $V_{\nu}$-module if and only if 
$$
\delta(\omega_i/\nu)=1\mbox{ and }\epsilon(\omega_i/\nu)=e(\omega_i/\nu)
$$
for all extensions $\omega_i$ of $\nu$ to $L$.
\end{Theorem}

Hagen Knaf proposed the following interesting question, asking for  a local form of the above theorem. Essential finite generation is defined in Section \ref{SecPrem}.

\begin{Question}\label{Knaf}(Knaf) Suppose that   $\omega$ is an extension of $\nu$ to $L$. Is $V_{\omega}$ essentially finitely generated over $V_{\nu}$ if and only if 
$$
\delta(\omega/\nu)=1\mbox{ and }\epsilon(\omega/\nu)=e(\omega/\nu)?
$$
\end{Question}

Knaf proved the ``only if''  direction of his question; his proof is reproduced in \cite[Theorem 4.1]{CN}.

The  ``if'' direction of the question is true if $L/K$ is normal or $\omega$ is the unique extension of $\nu$ to $L$ by \cite[Corollary 2.2]{CN}. In \cite[Theorem 1.3]{Hens}, it is shown that if $L$ is an inertial extension of $K$, then $V_{\omega}$ is essentially finitely generated over $V_{\nu}$.

The  ``if'' direction of the question is proven when $K$ is the quotient field of an excellent two-dimensional excellent local domain and $\nu$ dominates $R$ in \cite[Theorem 1.4]{CN}. 
The proof of \cite[Theorem 1.4]{CN} uses the existence of a resolution of excellent surface singularities (\cite{Lip} or \cite{CJS}) and local monomialization of defectless extensions of two dimensional excellent local domains (\cite[Theorem 3.7]{Ramif} and \cite[Theorem 7.3]{CP}). 

The ``if'' direction is proven when $K$ is an algebraic function field over a field $k$ of characteristic zero and $\nu$ is arbitrary in \cite[Theorem 1.3]{C3}.

We obtain the following theorem, which is proven in Section \ref{SecLocUnif}, as a consequence of Theorems \ref{TheoremA} and \ref{TheoremC}. 

\begin{Theorem}\label{TeoAbh} Let $K$ be an algebraic function field over a field $k$, and let $\nu$ be an Abhyankar valuation on $K$.  
Assume that $L$ is a finite  extension of $K$ and that $\omega$ is an extension of $\nu$ to $L$.  If  $\epsilon(\omega/\nu)=e(\omega/\nu)$, then $V_\omega$ is essentially finitely generated over $V_\nu$.
\end{Theorem}

The defect $\delta(\omega/\nu)=1$ with the assumptions of Theorem \ref{TeoAbh} by Theorem \ref{TheoremK} as $\nu$ is an Abhyankar valuation.

Theorem \ref{TeoAbh} is proven in \cite[Theorem 1.5]{CN}, with the additional assumption that $V_\omega/m_\omega$ is separable over $k$. 
To prove the stronger Theorem \ref{TeoAbh}, we must only modify  the proof of \cite[Theorem 1.5]{CN} by observing that
the statement of  \cite[Theorem 7.2]{CN} is true without the  assumption that  $V_{\nu}/m_{\nu}$ is separable over $k$, using Theorems \ref{TheoremA} and \ref{TheoremB} of this paper in place of \cite[Theorem 1.1]{KK}.  

A complete positive answer to Question \ref{Knaf} has been found by Rankeya Datta at the time that this article was in press, in the article ``Essential finite generation of extensions of valuation rings'', arXiv:2101.08377.

\section{Notation}\label{SecPrem}
We will denote the non-negative integers by $\NN$ and $\ZZ_{>0}$ will denote the positive integers. 

If $R$ is a local ring we will denote its maximal ideal by $m_R$. A regular prime ideal in a Noetherian local ring is a prime ideal $P$ such that $R/P$ is a regular local ring. If $A$ is a domain then ${\rm QF}(A)$ will denote the quotient field of $A$. Suppose that $A$ is a subring of a ring $B$. We will say that $B$ is essentially finitely generated over $A$ (or that $B$ is essentially of finite type over $A$) if $B$ is  a localization of a finitely generated $A$-algebra. If $R$ and $S$ are local rings such that $R$ is a subring of $S$ and $m_S\cap R=m_R$ then we say that $S$ dominates $R$.

Suppose that $k$ is a field  and $K/k$ is an algebraic function field over $k$. An algebraic local ring of $K$ is a local domain which is essentially of finite type over $k$ and whose quotient field is $K$. A birational extension $R\rightarrow R_1$ of an algebraic local ring $R$ of $K$ is an algebraic local ring $R_1$ of $K$ such that $R_1$ dominates $R$.

If $\nu$ is a valuation of a field $K$, we will denote the valuation ring of $\nu$ by $V_{\nu}$ and its maximal ideal by $m_{\nu}$.
If a valuation ring $V_{\nu}$ dominates $A$ we will also say that $\nu$ dominates $A$. If $A$ is a subring of a valuation ring $V_{\nu}$, we will write $A_{\nu}$ for the localization of $A$ at $m_{\nu}\cap A$.

A valuation $\nu$ of a function field $K/k$ is a valuation of $K$ which is trivial on $k$.

A pseudo-valuation $\mu$ on a local domain $R$ is a prime ideal $P$ of $R$ and a valuation $\mu$ on the quotient field of $R/P$  which dominates $R/P$.
If $\mu$ is a pseudo-valuation on a domain $R$, we will write $\mu(f)=\infty$ if $f$ is in the kernel $P$ of the map from  $R$ to $V_{\mu}$. We will write $P=P(\mu)_R$.

\section{Extensions of Pseudo-Valuations}\label{SecPV}

Suppose that $T$ is a normal excellent local ring. Let $P(\omega)_T$ be a prime ideal of $T$ and $\omega$ be a rank one valuation of the quotient field of $T/P(\omega)_T$ which dominates $T/P(\omega)_T$.  The valuation $\omega$ induces a pseudo-valuation of $T$ (as defined in the above Section \ref{SecPrem}), where we define $\omega(f)=\omega(\overline f)$ if the class $\overline f$ of $f$ in $T/P(\omega)_T$ is nonzero, and define $\omega(f)=\infty$ if $f\in P(\omega)_T$. 
Let 
\begin{equation}\label{eqN1}
Q(T)=\left\{
\begin{array}{l}
 \mbox{ Cauchy sequences $\{f_n\}$ in $T$ such that for all $l\in \ZZ_{>0}$,}\\
 \mbox{ there exists $n_l\in \ZZ_{>0}$ such that $\omega(f_n)\ge l\omega(m_T)$ if $n\ge n_l$}
 \end{array}
 \right\}.
\end{equation}
We have that $Q(T)$ is a prime ideal in $\hat T$ and $Q(T)\cap T =P(\omega)_T$. There is a unique extension of $\omega$ to a valuation of the quotient field of $\hat T/Q(T)$ which dominates $\hat T/Q(T)$. It is an immediate extension (there is no extension of the value group or the residue fields of the valuation rings). Thus there is a unique extension of $\omega$ to a pseudo-valuation of $\hat T$ such that $P(\omega)_{\hat T}=Q(T)$.
We define
$$
\sigma(T)=\dim \hat T/Q(T).
$$

The objects $Q(T)$ and $\sigma(T)$ are defined in \cite{C1}, \cite{CG} and \cite{E1}. Concepts of this type are studied in \cite{GAST}.

The following lemma is proven in the case that $\omega$ is a rank 1 valuation dominating $T$ (and not just a pseudo-valuation) in \cite[Lemma 6.3]{C1}. The proof is essentially the same here, although we require a little more notation. 

\begin{Lemma}\label{LemmaD1} Let notation be as in this section. We further suppose that $V_{\omega}/m_{\omega}$ is an algebraic field extension of $T/m_T$. 
\begin{enumerate}
\item[1)] Let $I$ be a nonzero ideal in $T$ such that $I\not\subset P(\omega)_{T}$. Let $f\in I$ be such that $\omega(f)=\omega(I)$. Let
$$
J=\cup_{j=1}^{\infty}\left(P(\omega)_{ T}T[\frac{I}{f}]:I^jT[\frac{I}{f}]\right),
$$
which is the strict transform of the ideal $P(\omega)_{T}$ in $T[\frac{I}{f}]$.
Then $J$ is a prime ideal in $T[\frac{I}{f}]$, the map $T/P(\omega)_{T}\rightarrow   T[\frac{I}{f}]/J$ is birational ($T[\frac{I}{f}]/J$ is of finite type over $T/P(\omega)_{T}$ and both rings have the same quotient field) and there exists a maximal ideal $n$ of $T[\frac{I}{f}]$ containing $J$ such that $\omega$ dominates $(T[\frac{I}{f}]/J)_n$ and so $\omega$ is a pseudo-valuation on $T_1=T[\frac{I}{f}]_n$ with $P(\omega)_{T_1}=J_n$.

\item[2)] Suppose that $T_1$ is normal. Then $\sigma(T_1)\le \sigma(T)$.
\end{enumerate}
\end{Lemma}

\begin{proof} We first consider Statement 1). Let $\overline f$ be the class of $f$ in $T/P(\omega)_{T}$. Since $I\not\subset P(\omega)_{T}$ we have that $\overline f\ne 0$ and 
$$
T\left[\frac{I}{f}\right]/J=(T/P(\omega)_{T})\left[\frac{I(T/P(\omega)_{T})}{\overline f}\right]
$$
is a birational extension of $T/P(\omega)_{T}$ and all its elements  have nonnegative $\omega$-value. Let $n$ be the  prime ideal in $T[\frac{I}{f}]$ of elements of positive $\omega$-value.
$T/m_{T}\subset T[\frac{I}{f}]/n\subset V_{\omega}/m_{\omega}$ and $V_{\omega}/m_{\omega}$ is assumed to be algebraic over $T/m_{T}$. Thus
$T[\frac{I}{f}]/n$  is finite over $T/m_T$. By \cite[Theorem 15.6]{NM}, $n$ is a maximal ideal of $T[\frac{I}{f}]$.

We now establish statement 2). The completion of $T_1$ at it's maximal ideal is
$\hat T_1=\widehat{\hat T[\frac{I\hat T}{f}]_{\tilde n}}$ where $\tilde n=m_{\hat T_1}\cap \hat T[\frac{I\hat T]}{f}]$. Let 
$$
\tilde Q=\cup_{j=1}^{\infty}\left(Q(T)\hat T\left[\frac{I\hat T}{f}\right]_{\tilde n}:I^j\hat T\left[\frac{I\hat T}{f}\right]_{\tilde n}\right),
$$
the strict transform of $Q(T)$ in $\hat T[\frac{I\hat T}{f}]_{\tilde n}$. Since $I\not\subset P(\omega)_{\tilde T}$ we have that $\omega(f)=\infty$ if $f\in \tilde Q$. Thus $\tilde Q\subset Q(T_1)$. Now $\hat T/Q(\hat T)\rightarrow \hat T[\frac{I\hat T}{f}]_{\tilde n}/\tilde Q$ is birational and the residue field extension is finite, so by the dimension formula \cite[Theorem 15.6]{NM}, 
$$
\sigma(T)=\dim \hat T[\frac{I\hat T}{f}]_{\tilde n}/\tilde Q=\dim \hat T_1/\tilde QT_1
$$
 since completion is flat. Thus $\sigma (T)\ge \dim \hat T_1/Q(T_1)=\sigma(T_1)$.
\end{proof}

\section{Abhyankar valuations}\label{SecAbh}

Let $K$ be an algebraic function field of a field $k$ and 
 let $\nu$ be an Abhyankar valuation of $K/k$.
Since $\nu$ is Abhyankar,  there exists a transcendence basis 
$$
x_{0,1},\ldots,x_{0,r_0},x_{1,1},\ldots,x_{1,r_1},x_{2,1},\ldots,x_{t,1},\ldots,x_{t,r_t}\in V_{\nu}
$$
of $K$ over $k$ such that the classes $\overline x_{0,1},\ldots,\overline x_{0,r_0}$ of $x_{0,1},\ldots,x_{0,r_0}$ in $V_{\nu}/m_{\nu}$ are a transcendence basis of $V_{\nu}/m_{\nu}$ over $k$ and $\nu(x_{i,1}),\ldots,\nu(x_{i,r_i})$ is a $\ZZ$-basis of $\Gamma_i/\Gamma_{i-1}$ for $1\le i\le t$. In particular, 
$\nu(x_{i,1}),\ldots,\nu(x_{t,r_t})$ is a $\ZZ$-basis of $\Gamma_{\nu_i}$ for all $i$, and the set of classes of 
$x_{0,1},\ldots,x_{i-1,r_{i-1}}$ in $V_{\nu_i}/m_{\nu_i}$ is a transcendence basis of $V_{\nu_i}/m_{\nu_i}$ over $k$.

\begin{Definition}\label{Def1}
Suppose that $A$ is a regular algebraic local ring of $K$ which  is dominated by $\nu_i$. A regular system of parameters $z_{i,1},\ldots,z_{t,r_t}$ in $A$ such that $\nu_i(z_{j,1}),\ldots,\nu_i(z_{j,r_j})$ is a basis of $\Gamma_j/\Gamma_{j-1}$ for $i\le j\le t$ is called a very good regular system of parameters in $A$.
\end{Definition}

\begin{Definition}\label{Def2} Suppose that $A$ is a  regular algebraic local ring of $K$ which is dominated by $\nu_i$. Suppose that $z_{i,1},\ldots,z_{t,r_t}$ is a very good regular system of parameters in $A$. 
  A primitive transform along $\nu_i$, $A\rightarrow A_1$, is defined by
\begin{equation}\label{eq25}
z_{j,k}=z_{j,k}(1)z_{l,m}
\end{equation}
where $\nu_i(z_{j,k})>\nu_i(z_{l,m})$. We define $A_1=A[z_{j,k}(1)]_{\nu_i}$. Then $A_1$ is a regular algebraic local ring of $K$ which has the good regular system of parameters $z_{i,1}(1),\ldots,z_{t,r_t}(1)$ where $z_{\alpha,\beta}(1)=z_{\alpha,\beta}$ if $(\alpha,\beta)\ne (j,k)$.
\end{Definition}

The following proposition is \cite[Proposition 7.4]{CN} or \cite[Lemma 4.1]{CM}, which is a generalization of \cite[Theorem 2]{Z1}.

\begin{Proposition}\label{Perron} Suppose that $A$ is a regular algebraic local ring of $K$ which is dominated by $\nu_i$ and $z_{i,1},\ldots,z_{t,r_t}$ is a very good regular system of parameters in $A$. Suppose that 
$M_1=\prod_{\alpha,\beta}z_{\alpha,\beta}^{a_{\alpha,\beta}}$ and $M_2=\prod_{\alpha,\beta}z_{\alpha,\beta}^{b_{\alpha,\beta}}$ are monomials such that $\nu_i(M_1)<\nu_i(M_2)$. Then there exists a sequence of primitive transforms along $\nu_i$, 
$$
A\rightarrow A_1\rightarrow\cdots\rightarrow A_s,
$$
 such that $M_1$ divides $M_2$ in $A_s$.
\end{Proposition}

We remark that if $\nu_i(M_1)=\nu_i(M_2)$ in the statement of Proposition \ref{Perron}, then we have that $M_1=M_2$.

\begin{proof}
There exists a largest index $l$ such that $\prod_jz_{l,j}^{a_{l,j}}\ne \prod_jz_{l,j}^{b_{l,j}}$.
Then $\nu_i(\prod_jz_{l,j}^{a_{l,j}})<\nu_i(\prod_jz_{l,j}^{b_{l,j}})$.  By \cite[Theorem 2]{Z1}, there exists a sequence of primitive transforms  $A\rightarrow A_s$ along $\nu_i$ in the  variables  $z_{l,j}(m)$  from the regular parameters  of $A_m$ as $j$ varies, 
such that $\prod_jz_{l,j}^{a_{l,j}}$ divides $\prod_jz_{l,j}^{b_{l,j}}$ in $A_s$. Writing $M_1$ and $M_2$ in the regular parameters $z_{i,j}(s)$ of $A_s$ as
$$
M_1=\prod z_{i,j}(s)^{a_{i,j}(s)}\mbox{ and }M_2=\prod z_{i,j}(s)^{b_{i,j}(s)},
$$
we have that 
$$
M_2=\left(\prod_{i<l,j}z_{i,j}(s)^{b_{i,j}(s)}\right)\left(\prod_j z_{l,j}(s)^{b_{l,j}(s)}\right)
\left(\prod_{i>l,j}z_{i,j}(s)^{a_{i,j}(s)}\right)
$$
with $b_{l,j}(s)-a_{l,j}(s)\ge 0$ for all $j$ and for some $j$, $b_{l,j}(s)-a_{l,j}(s)> 0$.
Without loss of generality, this occurs for $j=1$. (If $b_{l,j}(s)=a_{l,j}(s)$ for all $j$, then $a_{l,j}=b_{l,j}$ for all $j$ in contradiction to our choice of $l$.)

Now perform a  sequence of primitive transforms $A_s\rightarrow A_m$ along $\nu_i$ defined by
$z_{l,1}(t)=z_{l,1}(t+1)z_{\alpha,\beta}(t+1)$ for $\alpha<l$ and $\beta$ such that $b_{\alpha,\beta}(t)<a_{\alpha,\beta}(t)$ where 
$$
M_1=\prod z_{\alpha,\beta}(t)^{a_{\alpha,\beta}(t)}\mbox{ and }M_2=\prod z_{\alpha,\beta}(t)^{b_{\alpha,\beta}(t)}
$$
to achieve that $M_1$ divides $M_2$ in $A_m$.
\end{proof}

\subsection{Construction of an algebraic local ring which is dominated by $\nu$ and has some good properties}

Let $L=k(x_{0,1},\ldots,x_{t,r_t})$. $L\rightarrow K$ is a finite field extension. Let $\omega=\nu|L$ and $\omega_i=\nu_i|L$.

Let 
$$
R_i=k[x_{0,1},\ldots,x_{t,r_t}]_{\omega_i}
=k(x_{0,1},\ldots,x_{i-1,r_{i-1}})[x_{i,1},\ldots,x_{t,r_t}]_{(x_{i,1},\ldots,x_{t,r_t})},
$$
a regular local ring. 

\begin{Lemma}\label{Lemma1} Suppose that $f\in V_{\omega_i}$. Then there exists a sequence of primitive transforms along $\omega_i$,
$$
R_i\rightarrow R^1_i=R_i[x_{i,1}(1),\ldots,x_{t,r_t}(1)]_{(x_{i,1}(1),\ldots,x_{r,t_r}(1))}
$$
such that $f=uM$ where $M$ is a monomial in $x_{i,1}(1),\ldots,x_{r,t_r}(1)$ and $u$ is a unit in $R^1_i$.
\end{Lemma}

It follows that $V_{\omega_i}/m_{\omega_i}=k(x_{0,1},\ldots,x_{i-1,r_{i-1}})$.

\begin{proof} Write $f=\frac{g}{h}$ with $g,h\in R_i$. Expand $g=\sum \alpha_{\lambda_{i,1},\ldots, \lambda_{t,r_t}}
x_{i,1}^{\lambda_{i,1}}\cdots x_{t,r_t}^{\lambda_{t,r_t}}$ and $h=\sum \beta_{\lambda_{i,1},\ldots, \lambda_{t,r_t}}
x_{i,1}^{\lambda_{i,1}}\cdots x_{t,r_t}^{\lambda_{t,r_t}}$ in $\hat R_i=\kappa[[x_{i,1},\ldots,x_{t,r_t}]]$ with
$$
\alpha_{\lambda_{i,1},\ldots, \lambda_{t,r_t}}, \beta_{\lambda_{i,1},\ldots, \lambda_{t,r_t}}\in 
\kappa=k(x_{0,1},\ldots,x_{0,r_0}).
$$
Let $I=(x_{i,1}^{\lambda_{i,1}}\cdots x_{t,r_t}^{\lambda_{t,r_t}}\mid \alpha_{\lambda_{i,1},\ldots, \lambda_{t,r_t}}\ne 0
)$ and $J=(x_{i,1}^{\lambda_{i,1}}\cdots x_{t,r_t}^{\lambda_{t,r_t}}\mid \beta_{\lambda_{i,1},\ldots, \lambda_{t,r_t}}\ne 0
)$ which are ideals in $R_i$. By Proposition \ref{Perron}, there exists a sequence of primitive transforms along $\omega_i$, $R_i\rightarrow R(1)=R_i[x_{i,1}(1),\ldots,x_{t,r_t}(1)]_{(x_{i,1}(1),\ldots,x_{r,t_r}(1))}$ and $d_{i,1},\ldots,d_{t,r_t}, e_{i,1},\ldots,e_{r,t_r}\in \NN$  such that $IR(1)=x_{i,1}(1)^{d_{i,1}}\cdots x_{t,r_t}(1)^{d_{r,t_r}}R(1)$
and $JR(1)=x_{i,1}(1)^{e_{i,1}}\cdots x_{t,r_t}(1)^{e_{r,t_r}}R(1)$ and $d_{k,l}\ge e_{k,l}$ for all $k,l$. Thus $g=x_{i,1}(1)^{d_{i,1}}\cdots x_{t,r_t}(1)^{d_{t,r_r}}\gamma$ and 
$$
h=x_{i,1}(1)^{e_{i,1}}\cdots x_{t,r_t}(1)^{e_{t,r_r}}\delta
$$
where $\gamma,\delta\in \widehat{R(1)}$ are units. Now $\gamma,\delta\in K\cap \widehat{R(1)}=R(1)$ are units
(by  \cite[Lemma 2]{Ab2}). Thus $f=x_{i,1}(1)^{d_{i,1}-e_{i,1}}\cdots x_{t,r_t}(1)^{d_{t,r_t}-e_{t,r_r}}\gamma\delta^{-1}$ has the desired form in $R(1)$.

\end{proof}

\begin{Remark}\label{Remark1} The sequence of primitive transforms $R_i\rightarrow R_i^1$ of Lemma \ref{Lemma1} induces a sequence of primitive transforms 
$$
R_1\rightarrow R_1^1=R_1[x_{i,1}(1),\ldots,x_{t,r_t}(1)]_{(x_{1,1},\ldots,x_{i-1,r_{i-1}},x_{i,1}(1),\ldots,x_{t,r_t}(1))}
$$
  along $\omega$ such that $(R_1^1)_{(x_{i,1}(1),\ldots,x_{t,r_t}(1))}=R_i^1$.
\end{Remark}

Now $V_{\omega_i}/m_{\omega_i}=k(x_{0,1},\ldots,x_{i-1,r_{i-1}})\rightarrow V_{\nu_i}/m_{\nu_i}$ is a finite algebraic extension since $\nu_i$ is Abhyankar. Thus there exist $\overline\alpha_{i,1},\ldots,\overline\alpha_{i,s_i}\in V_{\nu_i}/m_{\nu_i}$ such that 
$$
V_{\nu_i}/m_{\nu_i}=(V_{\omega_i}/m_{\omega_i})(\overline\alpha_{i,1},\ldots,\overline\alpha_{i,s_i}).
$$

Let $T_i$ be the integral closure of $V_{\omega_i}$ in $L$. There exists a maximal ideal $m$ of $T_i$ such that $(T_i)_m=V_{\nu_i}$ and so $V_{\nu_i}/m_{\nu_i}=T_i/m$. Let $\alpha_{i,1},\ldots,\alpha_{i,s}$ be lifts of $\overline\alpha_{i,1},\ldots,\overline \alpha_{i,s_i}$ to $T_i$. Let $f_j(x)\in V_{\omega_i}[x]$ be the minimal polynomial of $\alpha_{i,j}$ over $L$. Let $a_{i,j}^k$ be the coefficients of the $f_j(x)$. By Lemma \ref{Lemma1}, and Remark \ref{Remark1}, there exists a sequence of primitive transforms along $\omega$, $R_1\rightarrow R_1^1$, such that 
$$
a_{i,j}^k\in (R_1^1)_{\omega_i}=(R_1^1)_{(x_{i,1}(1),\ldots,x_{t,r_t}(1))}
$$
for all $j,k$.

We may thus construct a sequence of primitive transforms along $\omega$
$$
R_1\rightarrow R(1)\rightarrow \cdots\rightarrow R(t)
$$
such that 
$$
R(t)=k(x_{0,1},\ldots,x_{0,r_0})[x_{1,1}(t),\ldots,x_{t,r_t}(t)]_{(x_{1,1}(t),\ldots,x_{t,r_t}(t))},
$$
and
$$
(R(t))_{\omega_i}=k(x_{0,1},\ldots,x_{0,r_0},x_{1,1}(1),\ldots,x_{i-1,r_{i-1}}(t))[x_{i,1}(t),\ldots,x_{t,r_t}(t)]_{(x_{i,1}(t),\ldots,x_{t,r_t}(t))}
$$
for $1\le i\le t$ and $a_{i,j}^k\in (R(t))_{\omega_i}$ for all $j,k$ and $i$.

Replacing $R_1$ with $R(t)$, we may assume that $a_{i,j}^k\in (R_1)_{\omega_i}$ for all $i,j,k$. 

Let $B$ be the integral closure of $R=R_1$ in $K$.
Let $P_i(R)=P^{\nu}_i\cap R=P^{\omega}_i\cap R$. Then the localization 
$B_{P_i(R)}$ of the $R$-module $B$ at the prime ideal $P_i(R)$ of $R$
 is the integral closure in $K$ of $R_{\omega_i}=R_{P_i(R)}$ for all $i$. Thus $\alpha_{i,j}\in B_{P_i(R)}$ for all $j$. Now $B_{P_i(R)}$ is a finite $R_{\omega_i}$ module for all $i$ and $B_{P_i(R)}\cap P^{\nu}_i$ is a maximal ideal in $B_{P_i(R)}$. Thus there exists a maximal ideal $m_i$ in $B_{P_i(R)}$ such that $m_i=B_{P_i(R)}\cap m_{\nu_i}$. Thus $\widehat{B_{\nu_i}}$ is the $m_i$-adic completion of $B_{P_i(R)}$ with respect to $m_i$ and $B_{\nu_i}/m_{B_{\nu_i}}=B_{P_i(R)}/m_i$. We have
\begin{equation}\label{eq23}
\begin{array}{l}
V_{\omega_i}/m_{\omega_i}=R_{\omega_i}/m_{R_{\omega_i}}=k(\overline{x}_{0,1},\ldots,\overline{ x}_{i-1,r_{i-1}})
\subset k(\overline{x}_{0,1},\ldots,\overline{ x}_{i-1,r_{i-1}})(\overline{\alpha}_{i,1},\ldots,\overline{\alpha}_{i,s_i})\\
\subset B_{P_i(R)}/m_i=B_{\nu_i}/m_{B_{\nu_i}}\subset V_{\nu_i}/m_{\nu_i}=(V_{\omega_i}/m_{\omega_i})(\overline \alpha_{i,1},\ldots,\overline \alpha_{i,s_i})
\end{array}
\end{equation}
so $B_{P_i(R)}/m_i=B_{\nu_i}/m_{B_{\nu_i}}=V_{\nu_i}/m_{\nu_i}$ for all $i$.

Now $B_{P_i(R)}$ is a finite $R_{\omega_i}$-module which implies that 
$$
P_i(R)(B_{P_i(R)})_{m_i}=(x_{i,1},\ldots,x_{t,r_t})B_{m_i}
$$
 is an $m_i$-primary ideal in $B_{m_i}$, so that $x_{i,1},\ldots,x_{t,r_t}$ is a system of parameters in $B_{m_i}$, since
\begin{equation}\label{eq1}
\dim B_{m_i}=\dim R_{\omega_i}=r_i+\cdots+r_t.
\end{equation}
Thus $x_{i,1},\ldots,x_{t,r_t}$ is a system of parameters in $\widehat{B_{m_i}}=\widehat{B_{\nu_i}}$.

Let $k_i$ be a coefficient field of $\widehat{B_{\nu_i}}$ (so $k_i\cong V_{\omega_i}/m_{\omega_i}$).
$k_i[[x_{i,1},\ldots,x_{t,r_t}]]\subset \widehat{B_{\nu_i}}$ is a power series ring, by \cite[Corollary 2, page 293]{ZS2}, and $\widehat{B_{\nu_i}}$ is a finite module over $k_i[[x_{i,1},\ldots,x_{t,r_t}]]$ by \cite[Remark on page 293]{ZS2}. 
By our construction, $\nu_i(x_{i,1}),\ldots,\nu_i(x_{t,r_t})$ is a $\ZZ$-basis of $\Gamma_{\omega_i}$ and $k_i\cong V_{\omega_i}/m_{\omega_i}$ for $1\le i\le t$.
Let $P_{i+1}(B_{\nu_i})$ be the prime ideal
$$
P_{i+1}(B_{\nu_i})=(P^{\nu}_{i+1}V_{\nu_i})\cap B_{\nu_i}\subset (P^{\nu}_iV_{\nu_i})\cap B_{\nu_i}=m_{B_{\nu_i}}.
$$
Equation (\ref{eq1}) implies that 
$$
\dim B_{\nu_i}=r_i+\cdots+r_t=\dim R_{\omega_i}
$$
for all $i$.
Thus
\begin{equation}\label{eqN21}
\dim B_{\nu_i}/P_{i+1}(B_{\nu_i})=\dim B_{\nu_i}-\dim (B_{\nu_i})_{P_{i+1}(B_{\nu_i})}=\dim B_{\nu_i}-\dim B_{\nu_{i+1}}=r_i.
\end{equation}
Now $\widehat{B_{\nu_i}}$ is reduced and equidimensional of the same dimension $r_i+\cdots+r_t$ as $B_{\nu_i}$ and
\begin{equation}\label{eq2}
P_{i+1}(B_{\nu_i})\widehat{B_{\nu_i}}=I_1\cap\cdots\cap I_u
\end{equation}
where the $I_j$ are prime ideals in $\widehat{B_{\nu_i}}$ such that
$$
\dim \widehat{B_{\nu_i}}/I_j=\dim B_{\nu_i}/P_{i+1}(B_{\nu_i})=r_i
$$
for all $j$ by \cite[Scholie IV.7.8.3]{EGAIV.2}.
Define a prime ideal $Q_{i+1}=Q_{i+1}(\widehat B_{\nu_i})$ in $\widehat{B_{\nu_i}}$ of the Cauchy sequences $\{f_j\}$ in $B_{\nu_i}$ such that given $\gamma\in \Gamma_i/\Gamma_{i-1}$, there exists $j_0$ such that $\nu_i(f_j)>\gamma$ for $j\ge j_0$.
The ideal $Q_{i+1}(\widehat{B_{\nu_i}})$ is the ideal $Q(\widehat{B_{\nu_i}})$ defined in (\ref{eqN1}).

We have an injective finite map
\begin{equation}\label{eqN33}
k_i[[x_{i,1},\ldots,x_{i,r_i}]]\cong k_i[[x_{i,1},\ldots,x_{t,r_t}]]/Q_{i+1}\cap k_i[[x_{i,1},\ldots,x_{t,r_t}]]
\rightarrow \widehat{B_{\nu_i}}/Q_{i+1}
\end{equation}
so 
\begin{equation}\label{eqN20}
\dim \widehat{B_{\nu_i}}/Q_{i+1}=r_i.
\end{equation}
 Now $P_{i+1}(B_{\nu_i})\widehat{B_{\nu_i}}\subset Q_{i+1}$ implies one of the $I_j$ in (\ref{eq2}) is $Q_{i+1}$. Thus after possibly reindexing the $I_j$, we have that $Q_{i+1}=I_1$ and 
\begin{equation}\label{eq8}
P_{i+1}(B_{\nu_i})\widehat{B_{\nu_i}}=Q_{i+1}\cap I_2\cap \cdots\cap I_u.
\end{equation}

$\nu_i$ is composite with the valuation $\overline\nu_i$ with valuation ring 
$V_{\overline\nu_i}=(V_{\nu}/P^{\nu}_{i+1})_{P^{\nu}_i}$,  value group $\Gamma_{\overline\nu_i}=\Gamma_i/\Gamma_{i-1}$ and residue field $(V_{\nu}/P^{\nu}_i)_{P^{\nu}_i}$. $\overline\nu_i$ naturally induces a valuation on $\widehat{B_{\nu_i}}/Q_{i+1}$.

\begin{equation}\label{eq24}
\begin{array}{l}
\mbox{There exists a regular local ring $D$ which is essentially of finite type over $k$,}\\
\mbox{such that $B_{\nu_1}$ is a quotient of $D$.}
\end{array}
\end{equation}

 Let $P_i(D)$ be the preimage of $P_i^{\nu}$ under the composition $D\rightarrow B_{\nu_1}\rightarrow V_{\nu}$ and let $D_i=D_{P_i(D)}$.
Then  $D_i$ are regular local rings which are essentially of finite type over $k$ with regular parameters 
$x_{i,1},\ldots,x_{i,r_i},x_{i+1,1},\ldots,x_{t,r_t},y_1,\ldots,y_m$ in $D_i$ such that $B_{\nu_i}$ is a quotient of $D_i$
(we identify the $x_{i,j}$ with their image in $B$).

We have that $\nu_i$ induces a pseudo-valuation on $D_i$ with kernel $P_{t+1}(U_i)$.
Let 
\begin{equation}\label{eqN22}
A_i=\widehat{D_i}=k_i[[x_{i,1},\ldots,x_{i,r_i},x_{i+1,1},\ldots,x_{t,r_t},y_1,\ldots,y_m]].
\end{equation}
Let $Q=Q_{i+1}(A_i)$ be the preimage of $Q_{i+1}$ in $A_i$. Then $\overline\nu_i$ induces  a pseudo-valuation on $\widehat{B_{\nu_i}}$ with kernel $Q_{i+1}$ and induces a pseudo-valuation on $A_i$ with kernel $Q=Q_{i+1}(A_i)$.
$\overline\nu_i$ induces a pseudo-valuation on $D_i$, with kernel 
\begin{equation}\label{eq 15}
P_{i+1}(D_i)=Q(A_i)\cap D_i
\end{equation}
which is the preimage of $P_{i+1}^{\nu}$ in $D_i$.

More generally, suppose that 
\begin{equation}\label{eq17}
\begin{array}{l}
\mbox{$U_i$ is a regular local ring which is essentially of finite type over $k$ with quotient field $K$}\\
\mbox{such that $U_i$ dominates $D_i$ and the pseudo-valuation $\overline\nu_i$ dominates $U_i$.}
\end{array}
\end{equation}

We then have a natural homomorphism $\pi: U_i\rightarrow V_{\nu_i}$. For $j$ such that $i\le j\le t+1$, define
\begin{equation}\label{eq18}
P_j(U_i)=\pi^{-1}(P_j^{\nu}V_{\nu_i}).
\end{equation}
We have a chain of prime ideals
$$
P_{t+1}(U_i)\subset P_t(U_i)\subset \cdots \subset P_i(U_i)=m_{U_i}.
$$

Suppose that 
\begin{equation}\label{eq16}
x_{i,1},\ldots,x_{i,r_i},x_{i+1,1},\ldots,x_{i+1,r_{i+1}},\ldots,x_{t,r_t},y_1,\ldots,y_m
\end{equation}
is a regular system of parameters in $U_i$ such that $\nu_j(x_{j,1}),\ldots,\nu_j(x_{j,t_j})$ is a $\ZZ$-basis of $\Gamma_j/\Gamma_{j-1}$ for $i\le j\le t$. Such a regular system of parameters will be called a good regular system of parameters, or simply ``good regular parameters''. Sometimes we will abuse notation and allow a good system of parameters to be a permutation of an ordered list (\ref{eq16}).

\subsection{Transforms}

We define four types of transformations $U_i\rightarrow U(1)$ along $\nu_i$.
\vskip .1truein
\noindent{\bf Type $(1,j)$ with $i\le j$.} This is  a transform 
$$
x_{j,k}=\prod_{l=1}^{r_j}x_{j,l}(1)^{a_{k,l}},\mbox{ for }1\le k\le r_j
$$
where $a_{k,l}\in\NN$, ${\rm Det}(a_{k,l})=\pm 1$ and $\nu_j(x_{j,l}(1))>0$ for all $l$.
We define 
$$
U(1)=U_i[x_{j,1}(1),\ldots,x_{j,r_j}(1)]_{\nu_i}.
$$
\vskip .1truein
\noindent{\bf Type $(2,j)$ with $i\le j$.}  Suppose $u\in\{x_{j+1,1},\ldots,x_{t,r_t},y_1,\ldots,y_m\}$ with $\nu_i(u)\in \Gamma_{j+1}$. Let $a_1,\ldots,a_{r_j}\in \NN$. Define
$$
u=x_{j,1}^{a_1}\cdots x_{j,r_j}^{a_{r_j}}u(1)
$$
and $U(1)=U_i[u(1)]_{\nu_i}$.

\vskip .1truein
\noindent{\bf Type $(3,j)$ with $i\le j$.} Suppose $\nu_i(y_k)\in\Gamma_j$ and $\nu_j(y_k)=a_1\nu_j(x_{j,1})+\cdots+a_{r_j}\nu_j(x_{j,r_j})$ with $a_1,\ldots,a_{r_j}\in \NN$ and 
$$
\nu_i\left(\frac{y_k}{x_{j,1}^{a_1}\cdots x_{j,r_j}^{a_{r_j}}}\right)\ge 0.
$$
Define $y_k=x_{j,1}^{a_1}\cdots x_{j,r_j}^{a_{r_j}}y_k(1)$ and $U(1)=U_i[y_k(1)]_{\nu_i}$.

\vskip .1truein
\noindent{\bf Type $(4,j)$ with $i\le j$.} Suppose that $u\in\{y_1,\ldots,y_m\}$ with $\nu_i(u)\in \Gamma_j$. Suppose that $a_1,\ldots,a_{r_j}\in \NN$ are such that $\nu_i(u)>\nu_i(x_{j,1}^{a_1}\cdots x_{j,r_j}^{a_{r_j}})$. Define $u=x_{j,1}^{a_1}\cdots x_{j,r_j}^{a_{r_j}}u(1)$ and
$U(1)=U_i[u(1)]_{\nu_i}$.

In all four cases, $U(1)$ has a natural good system of regular parameters. 

\subsection{Formal Transforms}\label{SecFor}

We now construct sequences of formal transforms along $\overline\nu_i$
\begin{equation}\label{eqN2}
A_i=\widehat{U_i}\rightarrow A(1)\rightarrow\cdots\rightarrow A(l).
\end{equation}

We suppose that 
$$
x_{i,1},\ldots,x_{t,r_t},y_1,\ldots,y_e
$$
is a given good regular system of parameters in $U_i$. In $A(0)=A_i=\hat{U_i}$, we suppose that 
$$
x_{i,1}(0),\ldots,x_{t,r_t}(0),y_1(0),\ldots,y_e(0)
$$
is a regular system of parameters such that $x_{1,1}(0)=x_{1,1},\ldots,x_{t,r_t}(0)=x_{t,r_t}$ and $y_1(0),\ldots,y_e(0)$ may be formal (not in $U_i$). Suppose that we have inductively
defined $A(0)\rightarrow \cdots \rightarrow A(n)$, with a regular system of parameters
$$
x_{i,1}(n),\ldots,x_{t,r_t}(n),y_1(n),\ldots,y_e(n).
$$
\vskip .1truein
\noindent{\bf Formal transforms of type $(1,i)$.}
We define a formal transform $A(n)\rightarrow A(n+1)$ of type $(1,i)$ as follows. 
Set 
$$
x_{i,k}(n)=\prod_{l=1}^{r_i}x_{i,l}(n+1)^{a_{i,l}},\mbox{ for }1\le k\le r_i
$$
where $a_{i,l}\in\NN$, ${\rm Det}(a_{i,l})=\pm 1$ and $\overline \nu_i(x_{i,l}(n+1))>0$ for all $l$.
By Lemma \ref{LemmaD1}, since the ring $A(n)[x_{i,1}(n+1),\ldots,x_{i,r_i}(n+1)]$ is the blow up of an ideal in $A(n)$ generated by monomials in $x_{i,1}(n),\ldots,x_{i,r_i}(n)$ which is thus not contained in $Q(A(n))$,
the pseudo-valuation $\overline\nu_i$ extends uniquely to a pseudo-valuation which dominates a local ring 
$$
(A(n)[x_{i,1}(n+1),\ldots,x_{i,r_i}(n+1)])_{m_n}
$$
 where $m_n$ is a maximal ideal of  $A(n)[x_{i,1}(n+1),\ldots,x_{i,r_i}(n+1)]$, and extends uniquely to a pseudo-valuation which dominates it's completion
$$
A(n+1)=(A(n)[x_{i,1}(n+1),\ldots,x_{i,r_i}(n+1)]_{m_n})^{\widehat{}}.
$$
We extend $x_{i,1}(n+1),\ldots,x_{i,r_i}(n+1)$ to a regular system of parameters 
$$
x_{i,1}(n+1),\ldots,x_{i,r_i}(n+1),x_{i+1,1}(n+1),\ldots,x_{t,r_t}(n+1),y_1(n+1),\ldots,y_e(n+1)
$$
in $A(n+1)$ where
$x_{j,k}(n+1)=x_{j,k}(n)$ for $j>i$ and $y_1(n+1),\ldots,y_e(n+1)$ can be chosen arbitrarily to make a regular system of parameters.

\vskip .1truein
\noindent{\bf Formal transforms of type $(2,i)$.}
We define a formal transform $A(n)\rightarrow A(n+1)$ of type (2,i) as follows. Suppose that 
$u\in \{x_{i+1,1}(n),\ldots,x_{t,r_t}(n),y_1(n),\ldots,y_e(n)\}$ and $\overline\nu_i(u)=\infty$.
Suppose that  $a_1,\ldots,a_{r_j}\in \NN$.
Define $u=x_{i,1}(n)^{a_1}\cdots x_{i,r_i}(n)^{a_{r_j}}u'$.

By Lemma \ref{LemmaD1}, since the ring $A(n)[u']$ is the blow up of the ideal
$(x_{i,1}(n)^{a_1}\cdots x_{i,r_i}(n)^{a_{r_j}},u)$ in $A(n)$ 
which is  not contained in $Q(A(n))$,
the pseudo-valuation $\overline\nu_i$ extends uniquely to a pseudo-valuation which dominates a local ring 
$$
(A(n)[u'])_{m_n}
$$
 where $m_n$ is a maximal ideal of  $A(n)[u']$, and extends uniquely to a pseudo-valuation which dominates it's completion
$$
A(n+1)=\widehat{A(n)[u']_{m_n}}.
$$

We define a regular system of parameters
$$
x_{i,1}(n+1),\ldots,x_{i,r_i}(n+1),x_{i+1,1}(n+1),\ldots,x_{t,r_t}(n+1),y_1(n+1),\ldots,y_e(n+1)
$$
in $A(n+1)$. 

 If $u\in \{y_1(n),\ldots,y_e(n)\}$, then 
$x_{j,k}(n+1)=x_{j,k}(n)$ for $j\ge i$ and $y_1(n+1),\ldots,y_e(n+1)$ can be chosen arbitrarily to make a regular system of parameters.
If $u=x_{a,b}(n)$ for some $a>i$ and $b$, then 
$$
x_{j,k}(n+1)=
\left\{\begin{array}{ll}
u'&\mbox{ if }(j,k)= (a,b)\\
x_{j,k}(n)&\mbox{ if }(j,k)\ne (a,b)\\
\end{array}\right.
$$
and $y_1(n+1),\ldots,y_e(n+1)$ can be chosen arbitrarily to make a regular system of parameters.

\vskip .1truein
\noindent{\bf Formal transforms of type $(3,i)$.}

We define a formal transform $A(n)\rightarrow A(n+1)$ of type $(3,i)$ as follows. 
Suppose $\overline \nu_i(y_1(n))\neq\infty$,  and $\overline \nu_i(y_1(n))=a_1\overline \nu_i(x_{i,1}(n))+\cdots+a_{r_j}\overline \nu_i(x_{i,r_i})$ with $a_1,\ldots,a_{r_j}\in \NN$ and 
$$
\overline \nu_i\left(\frac{y_1(n)}{x_{i,1}(n)^{a_1}\cdots x_{i,r_i}(n)^{a_{r_j}}}\right)\ge 0.
$$

Define $y_1(n)=x_{i,1}(n)^{a_1}\cdots x_{i,r_i}(n)^{a_{r_i}}y_1'$.

By Lemma \ref{LemmaD1}, since the ring $A(n)[y_1']$ is the blow up of the ideal 
$(x_{i,1}(n)^{a_1}\cdots x_{i,r_i}(n)^{a_{r_i}},y_1(n))$
in $A(n)$  which is  not contained in $Q(A(n))$,
the pseudo-valuation $\overline\nu_i$ extends uniquely to a pseudo-valuation which dominates a local ring 
$$
(A(n)[y_1'])_{m_n}
$$
 where $m_n$ is a maximal ideal of  $A(n)[y_1']$, and extends uniquely to a pseudo-valuation which dominates it's completion
$$
A(n+1)=\widehat{A(n)[y_1']_{m_n}}.
$$

We define a regular system of parameters
$$
x_{i,1}(n+1),\ldots,x_{i,r_i}(n+1),x_{i+1,1}(n+1),\ldots,x_{t,r_t}(n+1),y_1(n+1),\ldots,y_e(n+1)
$$
in $A(n+1)$ where
$x_{j,k}(n+1)=x_{j,k}(n)$ for $j\ge i$ and $y_1(n+1),\ldots,y_e(n+1)$ can be chosen arbitrarily to make a regular system of parameters.

\vskip .1truein
\noindent{\bf Formal transforms of type $(4,i)$.}

We define a formal transform $A(n)\rightarrow A(n+1)$ of type (4,i) as follows. Suppose that 
$\overline \nu_i(y_1(n))\neq\infty$
and  $a_1,\ldots,a_{r_j}\in \NN$ are such that $\nu_i(y_1(n))>\nu_i(x_{i,1}(n)^{a_1}\cdots x_{i,r_i}(n)^{a_{r_j}})$. 
Define $y_1(n)=x_{i,1}(n)^{a_1}\cdots x_{i,r_i}(n)^{a_{r_j}}y_1'$.

By Lemma \ref{LemmaD1}, since the ring $A(n)[y_1']$ is the blow up of the ideal
$(x_{i,1}(n)^{a_1}\cdots x_{i,r_i}(n)^{a_{r_j}},y_1(n))$ in $A(n)$ 
which is  not contained in $Q(A(n))$,
the pseudo-valuation $\overline\nu_i$ extends uniquely to a pseudo-valuation which dominates a local ring 
$$
(A(n)[y_1'])_{m_n}
$$
 where $m_n$ is a maximal ideal of  $A(n)[y_1']$, and extends uniquely to a pseudo-valuation which dominates it's completion
$$
A(n+1)=\widehat{A(n)[y_1']_{m_n}}.
$$

We define a regular system of parameters
$$
x_{i,1}(n+1),\ldots,x_{i,r_i}(n+1),x_{i+1,1}(n+1),\ldots,x_{t,r_t}(n+1),y_1(n+1),\ldots,y_e(n+1)
$$
in $A(n+1)$ where
$x_{j,k}(n+1)=x_{j,k}(n)$ for $j\ge i$ and $y_1(n+1),\ldots,y_e(n+1)$ can be chosen arbitrarily to make a regular system of parameters.

\begin{Remark} In a sequence (\ref{eqN2}) of formal transforms, we have that $\sigma(A(n))=\dim B_{\nu_i}$ for all $n$.

This follows since for all $n$, we have that 
$$
\dim B_{\nu_i}/P_{i+1}(B_{\nu_i})=\sigma(A(0))\ge   \sigma(A(n))\ge {\rm rrank }\,\,\,\overline\nu_i=\dim B_{\nu_i}/P_{i+1}(B_{\nu_i})
$$
 by (\ref{eqN20}), (\ref{eqN21}), Lemma \ref{LemmaD1} and \cite[Proposition 2]{Ab1} or Proposition 2 page 331 \cite{ZS2}.
\end{Remark}

\vskip .2truein
We will call a regular system of parameters 
\begin{equation}\label{eqN3}
x_{i,1}(n),\ldots,x_{t,r_t}(n),y_1(n),\ldots,y_e(n)
\end{equation}
as constructed in the above sequence of formal transforms (\ref{eqN2}) a good regular system of parameters in $A(n)$. We will find it convenient to permute the variables in a good regular system of parameters (\ref{eqN3}) and write it as
\begin{equation}\label{eqN4}
x_1(n),\ldots,x_r(n),z_1(n),\ldots,z_m(n),x_{i+1,1}(n),\ldots,x_{t,r_t}(n),w_1(n),\ldots,w_l(n)
\end{equation}
where 
$$
x_1(n)=x_{i,1}(n),\ldots,x_r(n)=x_{i,r_i}(n),
$$
$z_1(n),\ldots,z_m(n),w_1(n),\ldots,w_l(n)$ is a permutation of $y_1(n),\ldots,y_e(n)$
and 
$$
w_1(n),\ldots,w_l(n)\in Q(A(n)).
$$
We remind the reader that $x_{i+1,1}(n),\ldots,x_{t,r_t}(n)\in Q(A(n))$. We will also call (\ref{eqN4}) a good system of parameters in $A(n)$.

\begin{Remark}\label{RemarkN24} If $U(n)\rightarrow U(n+1)$ is a tranform or $A(n)\rightarrow A(n+1)$ is a formal transform then we have 
$$
k_i[[x_{i,1}(n),\ldots,x_{i,r_i}(n)]]\subset k_i[[x_{i,1}(n+1),\ldots,x_{i,r_i}(n+1)]].
$$
\end{Remark}

\subsection{Setting up the reduction algorithm}
We now fix $i$. Let $U_i=D_i$ and $A_i=\hat D_i$ (\ref{eqN22}).
Let 
$$
x_1,\ldots,x_r,z_1,\ldots,z_m,x_{i+1,1},\ldots,x_{t,r_t},w_1,\ldots,w_l
$$
be a good system of regular parameters in $U_i$ of the form of (\ref{eqN4}), so that $w_1,\ldots,w_l\in P_{i+1}(U_i)$. 

Let $\overline z_j$ be the class of $z_j$ in $A_i/Q_{i+1}(A_i)\cong \widehat{B_{\nu_i}}/Q_{i+1}(\widehat{B_{\nu_i}})$ for $1\le j\le m$. The element $\overline z_j$ is integral over $C_i=k_i[[x_1,\ldots,x_r]]$ by (\ref{eqN33}). Thus there exists a relation $\overline z_j^{n_j}+a_{j,n_j-1}\overline z_j^{n_j-1}+\cdots+a_{j,0}=0$ with all $a_{j,t}\in C_i$. Thus $z_j^{n_j}+a_{j,n_j-1}z_{j}^{n_j-1}+\cdots+a_{j,0}\in Q_{i+1}(A_i)$. That is, $\overline\nu_i(z_j^{n_j}+a_{j,n_j-1}z_j^{n_j-1}+\cdots+a_{j,0})=\infty$. 
Set 
\begin{equation}\label{eqN24}
g^{(j)}(x_1,\ldots,x_r,x)=x^{n_j}+a_{j,n_j-1}x^{n_j-1}+\cdots+a_{j,0}
\end{equation}
for $1\le j\le m$ which are
 monic polynomials in $C_i[x]$ such that $\overline\nu_i(g^{(j)}(z_j))=\infty$. We  may assume that each $g^{(j)}(x)$ is an irreducible polynomial in $C_i[x]$.

 \begin{Lemma}\label{LemmaN31} There exists a sequence of transforms $U\rightarrow U(1)$ of types (1,i), (3,i) and (4,i) such that $U(0)$ has very good parameters
 $$
x_1(0),\ldots,x_r(0),z_1(0),\ldots,z_m(0),x_{i+1,1}(1),\ldots,x_{t,r_t}(1),w_1(0),\ldots,w_l(1)
$$
such that, with $A(0)=\widehat{U(0)}$,
\begin{enumerate}
\item[1)]  $x_{i+1,1}(1),\ldots,x_{t,r_t}(1),w_1(0),\ldots,w_l(1)\subset P_{i+1}(U(0))$.
\item[2)] There exist monic polynomials $f^{(j)}(x)\in k_i[[x_1(0),\ldots,x_r(0)]][x]$  such that 
$$
\overline \nu_i(f^{(j)}(z_j(0)))=\infty
$$
 for $1\le j\le m$.
\item[3)] Suppose that $\alpha\in k_i$ and $q\in \ZZ_{>0}$. Then there exists $\alpha'\in U(0)$ and $h\in m_{A(0)}^q$ such that
\begin{equation}\label{eq3}
\alpha=\alpha'+(x_1(0)\cdots x_r(0))^qh
\end{equation}
with $h\in A(0)$.
\end{enumerate}
\end{Lemma}

\begin{proof}  For each $j$, the monic polynomial $g^{(j)}(x)=x^{n_j}+a_{j,n_j-1}x^{n_j-1}+\cdots+a_{j,0}\in C_i[x]$ is irreducible. Further, we have that $\nu_i(a_{j,0})>0$ so $x$ divides the residue of $g^{(j)}(x)$ in $C_i/m_{C_i} [x]\cong k_i[x]$. Thus by Hensel's lemma (Theorem 14 on page 279 \cite{ZS2}),
all $a_{j,t}\in m_{C_i}$.  Since $\nu_i(z_j)\in \ZZ\nu(x_1)+\cdots+\ZZ\nu(x_r)$, by Proposition \ref{Perron}, there exists a transform of type (1,i), $U\rightarrow U'$, defined by 
$x_s=\prod x_k(0)^{c_{s,k}}$
such that performing the further transforms of types (i,4) and (i,3), setting
$z_j=x_1(0)\cdots x_r(0) z_j(0)$ for $1\le j\le m$ and $w_j=x_1(0)\cdots x_r(0) w_j(0)$ for $1\le j\le l$, we have that 
$$
b_{j,t}=\frac{a_{j,t}}{(x_1(0)\cdots x_r(0))^{n_j-t}}\in k_i[[x_1(0),\ldots, x_r(0)]]
$$
for $0\le t\le n_j-1$. Thus setting $f^{(j)}(x)=x^n+b_{j,n_j-1}x^{n_j-1}+\cdots+b_{j,0}$ we have that $f^{(j)}(x)\in k_i[[x_1(0),\ldots, x_r(0)]][x]$  and $\overline \nu_i(f^{(j)}(z_j(0)))=\infty$ for all $j$.
\end{proof}

Replacing $U_i$ with $U(0)$, we may suppose that the good  properties of  Lemma \ref{LemmaN31} hold in $U_i$ and $A_i$.

Suppose that $m\ge 1$, and let $z=z_1$. In the algorithms of Subsections \ref{MA} - \ref{AlgSeq}, we will show that we can construct a sequence of transforms of types (i,1) and (i,3) giving a reduction of $m$.
In our reduction algorithm, none of the variables 
$$
z_2,\ldots,z_m,x_{i+1,1},\ldots,x_{t,r_t},w_1,\ldots,w_l
$$
will be effected, so we need only keep track of the change in the first variables $x_1,\ldots,x_r,z$.

In Theorem \ref{Theorem1}, we will prove that we can apply this algorithm $m$ times to obtain the condition that $Q_{i+1}$ is a regular prime. 



\subsection{The reduction algorithm}\label{MA}

Let $z=z_1$ and let 
$$
f(x)=f^{(1)}(x)=x^n+a_{n-1}x^{n-1}+\cdots+a_0.
$$
 Set $\mu=\mbox{ord }f(0,\ldots,0,x)$. We have that $1\le \mu\le n$.
If $\mu=1$ we replace $z$ with $f(x_1,\ldots,x_r,z)$, and obtain a reduction in $m$.  So suppose that $\mu\ge 2$. Set $a_n=1$ and let $\rho=\min\{\overline\nu_i(a_jz^j)\}$.  Write
$$
a_0+a_1z+\cdots+a_{n-1}z^{n-1}+z^n=a_{i_1}z^{i_1}+\cdots+a_{i_s}z^{i_s}+a_{i_{s+1}}z^{i_{s+1}}+\cdots+a_{i_{n+1}}z^{i_{n+1}}
$$
where $\rho=\overline\nu_{i}(z^{i_1})=\cdots=\overline \nu_i(a_{i_s}z^{i_s})$, with $i_1<\cdots<i_s$ and
$\rho<\overline\nu_i(a_{i_j}z^{i_j})\le \overline\nu_i(a_{i_{j+1}}z^{i_j+1})$ if $j>s$.

There exist $d_i\in \ZZ$ such that $\overline\nu_i(z)=\sum_{i=1}^rd_i\overline\nu_i(x_i)$. There exists a 
formal transform $A_i\rightarrow A(1)$ along $\overline\nu_i$ of type (1,i), 
$x_j=\prod_k x_k(1)^{b_{jk}}$ for $1\le j\le r$, as in Lemma \ref{Lemma1},  such that $A(1)$ has regular parameters $x_1(1),\ldots,x_r(1),z,\ldots$ such that  if $a_{i_j}$ is nonzero, then $a_{i_j}$ is a unit $\overline a_{i_j}$ in $k_i[[x_1(1),\ldots,x_r(1)]]$ times a monomial in $x_1(1),\ldots,x_r(1)$. Further, by Proposition \ref{Perron}, $x_1^{d_1}\cdots x_r^{d_r}$ is a monomial $x_1(1)^{e_1}\cdots x_r(1)^{e_r}$ (all $e_i$ are nonnegative).

Now perform the formal  transform $A(1)\rightarrow A(2)$ along $\overline\nu_i$ of type (3,i) defined by
\begin{equation}\label{eq4}
z=x_1(1)^{e_1}\cdots x_r(1)^{e_r}\tilde z_1
\end{equation}
so that $\overline\nu_i(\tilde z_1)=0$.   Now $A(2)/m_{A(2)}\subset V_{\nu_i}/m_{\nu_i}\cong k_i$ so $k_i$ continues to be a coefficient field of $A(2)$ and there exists a unit $\alpha\in k_i$ such that setting 
\begin{equation}\label{eq5}
z_1=\tilde z_1-\alpha,
\end{equation}
$x_1(1),\ldots,x_r(1),z_1,\ldots$ are good regular parameters in $A(2)$. 

\begin{equation}\label{eq11}
\begin{array}{l}
\mbox{If $\overline\nu_i(z_1)=\infty$ we terminate the algorithm.}\\
\mbox{Since $z_1\in Q_{i+1}(2)=Q_{i+1}(A(2))$, we have a reduction of $m$ in $A(2)$.}
\end{array}
\end{equation}

Suppose that $\overline\nu_i(z_1)\ne\infty$, We have that
$$
\begin{array}{lll}
f(x_1,\ldots,x_r,z)&=&a_{i_1}z^{i_1}+\cdots+a_{i_s}z^{i_s}+a_{i_{s+1}}z^{i_{s+1}}+\cdots+a_{i_{n+1}}z^{i_{n+1}}\\
&=&x_1(1)^{g_1}\cdots x_r(1)^{g_r}\left(\overline a_{i_1}(z_1+\alpha)^{i_1}+\cdots+\cdots+\overline a_{i_s}(z_1+\alpha)^{i_s}\right)\\
&+&x_1(1)^{g_{1,i_{s+1}}}\cdots x_r(1)^{g_{r,i_{s+1}}}\overline a_{i_{s+1}}(z_1+\alpha)^{i_{s+1}}+\cdots\\
&+&x_1(1)^{g_{1,i_{n+1}}}\cdots x_r(1)^{g_{r,i_{n+1}}}\overline a_{i_{n+1}}(z_1+\alpha)^{i_{n+1}}.
\end{array}
$$
Now perform a formal transform along $\overline\nu_i$ $A(2)\rightarrow A(3)$ of type (1,i) in $x_1(1),\ldots,x_r(1)$ so that $x_1(1)^{g_1}\cdots x_r(1)^{g_r}$ divides $x_1(1)^{g_{1,i_j}}\cdots x_r(1)^{g_{r,i_j}}$ for all $j$. Setting 
$$
f_1(x_1(1),\ldots,x_r(1),z_1)=\frac{f}{x_1(1)^{g_1}\cdots x_r(1)^{g_r}}\in k_i[[x_1(1),\ldots,x_r(1)]][z_1]
$$
 we have that $\overline\nu_i(f_1)=\infty$. We expand
\begin{equation}\label{eq7}
f_1=
\overline a_{i_1}(z_1+\alpha)^{i_1}+\cdots+\overline a_{i_s}(z_1+\alpha)^{i_s}+x_1(2)\cdots x_r(2)\Omega
\end{equation}
with $\Omega\in k_i[[x_1(2),\ldots,x_r(2)]][z_1]$.
We have that 
$$
f_1(0,\ldots,0,z_1)=\tilde a_{i_1}(z_1+\alpha)^{i_1}+\cdots+\tilde a_{i_s}(z_1+\alpha)^{i_s}\in k_i[z_1].
$$
where $\tilde \alpha_{i_j}$ is the residue of $\overline a_{i_j}$ in $k_i\cong A(2)/m_{A(2)}$.
 We have that $1\le \mbox{ord }f_1(0,\ldots,0,z_1)\le \mu$ (since $\overline \nu_i(f_1)=\infty$).

\begin{equation}\label{reduction}
 \begin{array}{l}
 \mbox{If $\mbox{ord }f_1(0,\ldots,0,z_1)<\mu$ then we have a reduction.
Go back to the beginning of }\\
\mbox{Subsection \ref{MA}  with $z$   replaced by $z_1$, $A_i$ replaced with $A(3)$, $C_i$ replaced with  }\\
 \mbox{$k_i[[x_1(2),\ldots,x_r(2))]]$, $f$ replaced with $f_1$ and $\mu$ replaced with 
 $\mbox{ord }f_1(0,\ldots,0,z_1)$.}
 \end{array}
 \end{equation} 
 Of course it may now be  that $z$ is formal.

If $\mbox{ord }f_1(0,\ldots,0,z_1)=\mu$,  then $i_s=\mu$, $a_{i_s}$ is a unit in $C_i$ (since $\mbox{ord }f(0,\ldots,0,x)=\mu$) and $f_1(0,\ldots,0,z_1)= a_{i_s}z_1^{\mu}$. Thus
$$
f_1(0,\ldots,0,z_1-\alpha)=\tilde a_{i_1}z_1^{i_1}+\cdots+\tilde a_{i_s}z_1^{i_s}=\tilde a_{i_s}(z_1-\alpha)^{i_s}.
$$
So, $i_1=0$ and $a_{i_1}\ne 0$. Thus 
$$
\begin{array}{lll}
\overline \nu_i(z)&=&\frac{1}{\mu}\overline \nu_i(a_{i_1})=\frac{1}{\mu}\overline\nu_i(a_0)\\
&\in&\frac{1}{\mu}\left(\NN\nu_i(x_1)+\cdots+\NN\nu_i(x_r)\right)\cap \left(\ZZ\nu_i(x_1)+\cdots+\ZZ\nu_i(x_r)\right)\\&=&\NN\nu_i(x_1)+\cdots +\NN\nu_i(x_r).
\end{array}
$$
Thus there exist $l_1,\ldots,l_r\in \NN$ such that $\overline\nu_i(z)=\nu_i(x_1^{l_1}\cdots x_r^{l_r})$, and so
there  exists $\lambda\in k_i\cong V_{\nu_i}/m_{\nu_i}$ such that $\overline \nu_i(z-\lambda x_1^{l_1}\cdots x_r^{l_r})>\nu_i(z)$. Set $z_1=z-\lambda x_1^{l_1}\cdots x_r^{l_r}$ and
  $f_1(x_1,\ldots,x_r,z_1)=f(x_1,\ldots,x_r,z)$ (this $f_1$ is different from the $f_1$ of (\ref{eq7})). We have that $\mbox{ord }f_1(0,\ldots,0,z_1)=\mu$.
\begin{equation}\label{Startagain}
\begin{array}{l}
\mbox{Go back to the beginning of Subsection \ref{MA}  with $z$ replaced by $z_1$
 and $f$ replaced with $f_1$,}\\
 \mbox{and run the reduction algorithm in $A_i$.}
 \end{array}
 \end{equation} 
Of course it may now be  that $z$ is formal.

\subsection{Termination of the reduction algorithm}\label{termination}
We either terminate after 
a finite number of iterations of the reduction algorithm, or after a finite number $s$ of reductions (\ref{reduction}) of $\mu$, we never find a reduction in $\mu$ after that, performing the operation of  equation (\ref{Startagain}) infinitely many times.

We thus construct a sequence
$$
A_i=G(0)\rightarrow G(1)\rightarrow \cdots \rightarrow G(s)
$$
where each $G(j)\rightarrow G(j+1)$ is an iteration of the reduction algorithm, culminating in the reduction step (\ref{reduction}) and such that the algorithm proceeding from $G(s)$ never produces a further  reduction in $\mu$. In each $G(j)$ we start with good parameters $\overline x_1(j),\ldots,\overline x_r(j),\overline z_j'$ and make a change of variables replacing $\overline z_j'$ with $\overline z_j$. This notation is chosen so that we may differentiate between different iterations of the reduction algorithm. After introducing this notation and explaining the construction of $G(j-1)\rightarrow G(j)$, we will explain the three possibilities that the algorithm from $G(s)$
can take.

In $G(0)=A_i$, we start with $\overline z_0'\in U_i$ with $\overline\nu_i(\overline z_0')<\infty$, and  $\overline x_1(0),\ldots, \overline x_r(0),\overline z_0'$ which are the first part of a good regular system of parameters in $U_i$. If $s>0$, we  make a change of variables 
\begin{equation}\label{eq9}
\overline z_0=\overline z_0'-\sum \lambda_{0,l}\overline x_1(0)^{u_{1,l}(0)}\cdots \overline x_r(0)^{u_{r,l}(0)}
\end{equation}
with $\lambda_{0,l}\in k_i$ and the sum is finite. We then apply the reduction algorithm to $\overline z_0$, to construct $G(0)\rightarrow G(1)$. 

Each $G(j-1)\rightarrow G(j)$ terminates with a new variable $\overline z_j'$, which is derived from $\overline z_{j-1}$ in the reduction algorithm
(these are the variables named $z_1$ and $z$ respectively in the reduction step (\ref{reduction})). The reduction algorithm gives $\overline x_1(j),\ldots,\overline x_r(j),\overline z_j'$ which are part of a regular system of parameters in $G(j)$. 
If $j<s$, we  make a change of variables 
\begin{equation}\label{eq10}
\overline z_j=\overline z_j'-\sum \lambda_{j,l}\overline x_1(j)^{u_{1,l}(j)}\cdots \overline x_r(j)^{u_{r,l}(j)}
\end{equation}
with $\lambda_{j,l}\in k_i$ and the sum is finite, and perform the reduction algorithm on $\overline z_j$ in $G(j)$.

We thus construct $\overline z_0,\overline z_1,\ldots,\overline z_s'$ by performing the reduction algorithm of subsection \ref{MA}, each time obtaining a reduction in $\mu$, giving a formal sequence of  transforms 
$$
A_i=G(0)\rightarrow\cdots\rightarrow G(s)
$$
along $\overline\nu_i$.

From $G(s)$ there are three possible paths that the algorithm can take. 
\begin{enumerate}
\item[1)] After a change of variables of the from of (\ref{eq10}) in $\overline z_s'$ we obtain $\overline z_s$ such that $\overline\nu_i(\overline z_s)=\infty$ or there exists $g(\overline x_1(s),\ldots,\overline x_r(s),\overline z_s)\in k_i[[\overline x_1(s),\ldots,\overline x_r(s)]][\overline z_s]$ such that $\overline\nu_i(g)=\infty$ and $\mbox{ord }g(0,\ldots,0,\overline z_s)=1$. In both cases we terminate the algorithm with a reduction in $m$.
\item[2)] We make a change of variables of the form of (\ref{eq10}) in $\overline z_s'$ to obtain $\overline z_s$. Then we perform a formal transform of type (i,1) 
$\overline x_i(s)=\prod_k\overline x_k(s,1)^{b_{i,k}(s)}$ for $1\le i\le r$ followed by a formal transform of type (i,3) of (\ref{eq4})
$$
\overline z_s=\overline x_1(s,1)^{e_1}\cdots \overline x_r(s,1)^{e_r}\tilde z_s
$$
followed by the change of variables (\ref{eq5}),
$$
\overline z_{s+1}'=\tilde z_s-\alpha_{s+1}
$$
with $\alpha_{s+1}\in k_i$ such that $\overline \nu_i(\overline z_{s+1}')=\infty$, terminating the algorithm in (\ref{eq11}) with a reduction in $m$.
\item[3)] We perform the operation of (\ref{Startagain}) infinitely many times, never terminating the algorithm.
\end{enumerate}

Let us assume that this third case occurs. We set $\overline z_s=\overline z_s'$.
We 
 repeat the algorithm of subsection \ref{MA} infinitely many times in $G(s)$, each time culminating in step (\ref{Startagain}), constructing 
$\alpha_{s+i}\overline x_1(s)^{g_1(s+i)}\cdots \overline x_r(s)^{g_r(s+i)}$ for $i\ge 0$ with $\alpha_{s+i}\in k_i$ and $g_1(s+i),\ldots,g_r(s+i)\in \NN$ such that for all $i$,
$$
\overline\nu_i(\overline x_1(s)^{g_1(s+i+1)}\cdots \overline x_r(s)^{g_r(s+i+1)})>\overline\nu_i(\overline x_1(s)^{g_1(s+i)}\cdots \overline x_r(s)^{g_r(s+i)})
$$
and the sequence 
$$
\overline z_{s+i+1}=\overline z_{s+i}-\alpha_{s+i}\overline x_1(s)^{g_1(s+i)}\cdots \overline x_r(s)^{g_r(s+i)}
$$
 satisfies 
 $$
 \overline\nu_i(\overline z_{s+i})=\overline\nu_i(\overline x_1(s)^{g_1(s+i)}\cdots \overline x_r(s)^{g_r(s+i)}),\,\,\,
 \overline\nu_i(\overline z_{s+i+1})>\overline\nu_i(\overline z_{s+i}).
 $$
 Since $\nu_i(\overline x_1(s)^{g_1(s+i)}\cdots \overline x_r(s)^{g_r(s+i)})$ is an increasing sequence in the semigroup $\NN\nu_i(\overline x_1(s))+\cdots+\NN\nu_i(\overline x_r(s))$, we have that $\nu_i(\overline x_1(s)^{g_1(s+i)}\cdots \overline x_r(s)^{g_r(s+i)})\mapsto \infty$ as $i\mapsto \infty$. Thus $\overline \nu_i(\overline z_{s+i})\mapsto \infty$ as $i\mapsto \infty$. Let $z_{\infty}$ be the limit in $G(s)$ of the Cauchy sequence $\{\overline z_{s+i}\}$. We have that the regular parameter $z_{\infty}$ satisfies $\overline\nu_i(z_{\infty})=\infty$, so $z_{\infty}\in Q(s)=Q_{i+1}(G(s))$ and so we have a reduction of $m$ in $G(s)$.
 Thus the third case produces a change of variables
 \begin{equation}\label{eqV1}
z_{\infty}=\overline z_s'-\sum \lambda_{s,l}\overline x_1(s)^{u_{1,l}(s)}\cdots \overline x_r(s)^{u_{r,l}(s)}
\end{equation}
in $G(s)$ with $\lambda_{s,l}\in k_i$ and where the sum is infinite.

 \subsection{The algorithm comes from an algebraic sequence of  transforms}\label{AlgSeq}
 
We use the notation of Section \ref{termination}.
 
 Recall that we constructed a sequence
 $$
 A_i=G(0)\rightarrow G(1)\rightarrow \cdots \rightarrow G(s)
 $$
 such that either we obtained  a reduction  of $m$ in $G(s)$, or we obtained a reduction after making  a final sequence
 $G(s)\rightarrow H$ where $H$ is a composition of a formal transform along $\nu_i$ of type (i,1) and a formal transform along $\nu_i$ of type (3,i).  
 
  We have good regular parameters
 $\overline x_1(j),\ldots,\overline x_r(j),\overline z_j',\ldots$ in $G(j)$ and in (\ref{eq10}) we make a change of variables obtaining a new system of good regular parameters 
 $$
 \overline x_1(j),\ldots,\overline x_r(j),\overline z_j,\ldots.
 $$
  Each $G(j)\rightarrow G(j+1)$ has a factorization 
 $$
 G(j)=E_j(0)\rightarrow E_j(1)\rightarrow E_j(2)\rightarrow E_j(3)=G(j+1).
 $$
  The parameters $\overline x_1(j),\ldots,\overline x_r(j),\overline z_j,\ldots$ are good regular parameters in in $E_j(0)$. The map $E_j(0)\rightarrow E_j(1)$ is  the formal transform along $\overline \nu_i$ of type (i,1) given by 
  $\overline x_i(j)=\prod_k\overline x_k(j,1)^{b_{i,k}(j)}$ for $1\le i\le r$. 
  Thus   $\overline x_1(j,1),\ldots,\overline x_r(j,1),\overline z_j,\ldots$ are good regular parameters in $E_j(1)$.  
  The map $E_j(1)\rightarrow E_j(2)$ is a formal transform along $\overline \nu_i$ of type (3,i), given by a substitution
  \begin{equation}\label{eqAg3}
  \overline z_j=\overline x_1(j,1)^{e_1(j)}\cdots\overline x_r(j,1)^{e_r(j)}\tilde z_j
  \end{equation}
  such that $\overline \nu_i(\tilde z_j)=0$. We then set 
  \begin{equation}\label{eqAg4}
  \overline z_{j+1}'=\tilde z_j-\alpha_j
  \end{equation}
   where $\alpha_j\in k_i$ so that
 \begin{equation}
 \overline x_1(j,1),\ldots,\overline x_r(j,1),\tilde z_{j+1}',\ldots
 \end{equation}
  are good regular parameters in $E_j(2)$.  The map $E_j(2)\rightarrow E_j(3)=G(j+1)$ is a formal transform of type (1,i) given by substitutions
  \begin{equation}\label{eqA5}
  \overline x_i(j,1)=\prod_k \overline x_k(j+1)^{c_{i,k}(j+1)}
  \end{equation}
   for $1\le i\le r$, resulting in the good regular parameters
  $\overline x_1(j+1),\ldots,\overline x_r(j+1), \overline z_{j+1}',\ldots$ in $E_j(3)=G(j+1)$.  
  
If we do not find a reduction of $m$ in $G(s)$, then we are in the case 2) discussed in Section \ref{termination}. In this case we have a sequence of maps $G(s)\rightarrow H_1\rightarrow H$, giving a reduction of $m$ in $H$.    
 
 We make a change of variables of the form of (\ref{eq10}) in $\overline z_s'$ to obtain $\overline z_s$, giving good regular parameters $\overline x_1(s),\ldots,\overline x_r(s),\overline z_s,\ldots$ in $G(s)$. The map $G(s)\rightarrow H_1$ is  a formal transform along $\overline \nu_i$ of type (i,1) given by a substitution  $\overline x_i(s)=\prod_k\overline x_k(s,1)^{b_{i,k}(j)}$ for $1\le i\le r$, giving good regular parameters $\overline x_1(s,1),\ldots,\overline x_r(s,1), \overline z_s,\ldots$ in $H_1$. The map $H_1\rightarrow H$ is a 
  formal transform of type (i,3) of  the form of (\ref{eq4}),
\begin{equation}\label{eqAg1}
\overline z_s=\overline x_1(s,1)^{e_1(s)}\cdots \overline x_r(s,1)^{e_r(s)}\tilde z_s
\end{equation}
followed by the change of variables (\ref{eq5}),
\begin{equation}\label{eqAg2}
\overline z_{s+1}'=\tilde z_s-\alpha_{s+1}
\end{equation}
with $\alpha_{s+1}\in k_i$ such that $\overline \nu_i(\overline z_{s+1}')=\infty$, terminating the algorithm in (\ref{eq11}) with a reduction in $m$, and giving good regular parameters $\overline x_1(s,1),\ldots,\overline x_r(s,1),\overline z_{s+1}',\ldots$ in $H$.

 We will show that there exists a sequence
 $$
 U_i=V(0)\rightarrow V(1)\rightarrow \cdots\rightarrow V(s)
 $$
 such that each sequence $V(j)\rightarrow V(j+1)$ is a sequence of  transforms along $\nu_i$ such that $G(j)=\widehat{V(j)}$ for all $j$. If we have a final sequence $G(s)\rightarrow H$, then we will construct a final sequence of transforms along $\nu_i$, $V(s)\rightarrow J$, such that $H=\hat J$. 
 


 We will construct the $V(j)$ by induction, so that $V(j)$ has good regular parameters 
 $$
 \overline x_1(j),\ldots,\overline x_r(j),z_j^*,\ldots
 $$
  such that 
 \begin{equation}\label{eq12}
 z_j^*=\overline z_j+(\overline x_1(j)\cdots \overline x_r(j))^{\sigma(j)}h_j
 \end{equation}
 with $h_j\in G(j)=\widehat{V(j)}$, and such that we can take $\sigma(j)$ arbitrarily large. 
 In particular, we have $\nu_i(z_j^*)=\overline\nu_i(\overline z_j)$.

 If $j=0$, we define $z_0^*$ as in (\ref{eq13}) below.
 
 Suppose that we have constructed $V(0)\rightarrow V(j)$ and $j<s$. We will construct $V(j)\rightarrow V(j+1)$.
 
 Define a transform $V(j)\rightarrow F(1)$ along $\nu_i$ of type (1,i) by the substitutions
  $\overline x_i(j)=\prod_k\overline x_k(j,1)^{b_{i,k}(j)}$ for $1\le i\le r$. 
  Thus   $\overline x_1(j,1),\ldots,\overline x_r(j,1),z_j^*,\ldots$ are good regular parameters in $F(1)$ and $\widehat{F(1)}=E_j(1)$.

 We have that  $\overline x_1(j)\cdots \overline x_r(j)$ is a monomial in $\overline x_1(j,1),\ldots,\overline x_r(j,1)$ in which all variables have positive exponents. Recall the substitution (\ref{eqAg3}) used to define the map $E_j(1)\rightarrow E_j(2)$. Taking $\sigma(j)$ sufficiently large in (\ref{eq12}), we have that
 $\overline x_1(j,1)^{e_1(j)}\cdots \overline x_r(j,1)^{e_r(j)}$ divides $(\overline x_1(j)\cdots \overline x_r(j))^{\sigma(j)}$ in $F(1)$ and
 $$
 \frac{(\overline x_1(j)\cdots \overline x_r(j))^{\sigma(j)}}{\overline x_1(j,1)^{e_1(j)}\cdots \overline x_r(j,1)^{e_r(j)}}
 =\overline x_1(j,1)^{v_1}\cdots \overline x_r(j,1)^{v_r}\in m_{F(1)}
 $$
 for some $v_1,\ldots,v_r\in \ZZ_{>0}$.
  Let $F(1)\rightarrow F(2)$ be the  transform along $\nu_i$ of type (3,i) defined by 
 $$
 z_j^*=\overline x_1(j,1)^{e_1(j)}\cdots \overline x_r(j,1)^{e_r(j)}\tilde v_{j+1}.
 $$
 Then 
 $$
 \tilde z_{j+1}=\tilde v_{j+1}-\overline x_1(j,1)^{v_1}\cdots \overline x_r(j,1)^{v_r}h_j
 $$
  and so $\widehat{F(2)}=E(2)$.

  The variable $\overline z_{j+1}'$ of $E(2)$ is defined by (\ref{eqAg4}).  By (\ref{eq3}), there exists $\alpha'\in U_i$ such that $\alpha_{j+1}=\alpha'+(\overline x_1(j,1)\cdots \overline x_r(j,1))^{\tau(j)}h$ where $h\in E(2)$ and $\tau(j)$ can be arbitrarily large. 
 
 Set $v_{j+1}'=\tilde v_{j+1}-\alpha'$.
 We have that $\overline x_1(j,1),\ldots,\overline x_r(j,1),\overline z_{j+1}'$ is part of a regular system of parameters in $E(2)$ and $\overline x_1(j,1),\ldots,\overline x_r(j,1),v_{j+1}'$ are part of a regular system of parameters in $F(2)$ such that 
 $$
 v_{j+1}'=\overline z_{j+1}'+(\overline x_1(j,1)\cdots\overline x_r(j,1))^{\beta(j)}g_j
 $$
  with $g_j\in E(2)$ and $\beta(j)$ can be arbitrarily large. 
  
  Finally define a transform along $\nu_i$ of type (1,i) $F(2)\rightarrow F(3)$ by the substutions
  $$
  \overline x_i(j,1)=\prod_k \overline x_k(j+1)^{c_{i,k}(j+1)}
  $$
   for $1\le i\le r$, so that $\widehat{F(3)}=E(3)$. We have thus constructed $V(j+1)=F(3)$.
 
 Now in the change of variables  of (\ref{eq10}), we have
 $$
 \overline z_{j+1}=\overline z_{j+1}'- \sum_l\lambda_{j+1,l}\overline x_1(j+1)^{u_{1,l}(j+1)}\cdots \overline x_r(j+1)^{u_{r,l}(j+1)}
 $$
  for $1\le l\le r$. We apply (\ref{eq3}) to find $\lambda_{j+1,l}'\in U_i$ such that 
 $$
 \lambda_{j+1,l}=\lambda_{j+1,l}'+(x_1(j+1)\cdots x_r(j+1))^{\omega(j+1)}h_{j+1,l}
 $$
 where $h_{j+1,l}\in G(j+1)$ and $\omega(j+1)$ can be arbitrarily large, and set
 \begin{equation}\label{eq13}
 z_{j+1}^*=v_{j+1}'-\sum_l\lambda_{j+1,l}'x_1(j+1)^{u_{1,l}(j+1)}\cdots x_r(j+1)^{u_{r,l}(j+1)}.
 \end{equation}
 
 The construction of the final sequence $V(s)\rightarrow J$ is a simplification of the argument for constructing $V(j-1)\rightarrow V(j)$.
 
 We may now prove  the following theorem.
 
 \begin{Theorem}\label{Theorem1} Let $U_i=D_i$. Then there exists a sequence of  transforms 
 of types (1,i), (2,i), (3,i) and (4,i),
 $U_i\rightarrow U(1)$ along $\nu_i$ such that $Q_{i+1}(1)=Q_{i+1}(\widehat{U(1)})$  is a regular prime ideal in $\widehat{U(1)}$; that is, $\widehat{U(1)}/Q_{i+1}(1)$ is a regular local ring.
  \end{Theorem}
  
  \begin{proof}  Let 
$$
x_1,\ldots,x_r,z_1,\ldots,z_m,x_{i+1,1},\ldots,x_{t,r_t},w_1,\ldots,w_l
$$
be a system of  good  regular parameters in $U_i$ of the form of (\ref{eqN4}), so that 
$$
x_{i+1,1},\ldots,x_{t,r_t},w_1,\ldots,w_l\in Q(U_i).
$$
  Suppose that $m\ge 1$. As shown above,
there exists a sequence of transforms along $\nu_i$ $U_i\rightarrow W(1)$ such that $\widehat{W(1)}$ has good regular parameters
$$
x_1(1),\ldots,x_r(1),z_2,\ldots,z_m,x_{i+1,1},\ldots,x_{t,r_t},w_1',w_1,\ldots,w_l
$$
where 
$$
x_{i+1,1},\ldots,x_{t,r_t},w_1',w_1,\ldots,w_l\in Q(W(1)).
$$
 Since $z_2,\ldots,z_m\in U_i$ are ``algebraic'', and they have  relations in $k_i[[x_{1}(1),\ldots,x_r(1)]]$ by Lemma \ref{LemmaN31},  we may continue the algorithm, finding a sequence of transforms along $\nu_i$
$U_i\rightarrow W(1)\rightarrow \cdots\rightarrow W(m)$ so that there are regular parameters 
$$
x_1(1),\ldots,x_r(1),x_{i+1,1},\ldots,x_{t,r_t},w_1',\ldots,w_m',w_1,\ldots,w_l
$$
in $\widehat{W(m)}$ such that 
$$
Q(W(m))=(x_{i+1,1},\ldots,x_{t,r_t},w_1',\ldots,w_m',w_1,\ldots,w_l).
$$
Thus the conclusions of the theorem hold in $W(m)$.
\end{proof}
 
 \subsection{Resolution in the smallest rank}

Suppose  that $U_i$ is as in (\ref{eq17}) and there exist 
$$
w_1,\ldots,w_l\in P=P_{i+1}(U_i)=Q_{i+1}(A_i)\cap U_i\subset U_i
$$
 and $\hat z_1,\ldots, \hat z_m\in Q=Q_{i+1}(A_i)\subset A_i=\widehat{U(i)}$ such that 
\begin{equation}\label{eq21}
\begin{array}{l}
Q=(\hat z_1,\ldots,\hat z_m,x_{i+1,1},\ldots, x_{t,r_t},w_1,\ldots,w_l)\\
\mbox{and}\\
x_1,\ldots,x_r,\hat z_1,\ldots,\hat z_m,
x_{i+1,1},\ldots, x_{t,r_t},w_1,\ldots,w_l\\
\mbox{is a good regular system of parameters in $A_i$.} 
\end{array}
\end{equation}
Good regular systems of parameters are defined at the beginning of SubSection \ref{SecFor}.

We will show that if $m\ge 1$, then there exists a sequence of  transforms $U_i\rightarrow U(1)$ along $\nu_i$ such that there exist an expression (\ref{eq21}) in $U(1)$ and $A(1)=\hat U(1)$ with a decrease of $m$ (and increase in $l$).

We will prove this by descending induction on $m$. By Theorem \ref{Theorem1}, we can assume that there is such an expression with  $m=\dim U_i-\dim B_{\nu_i}$ (and $l=0$).

From equation (\ref{eq8}), we have a reduced primary decomposition
$PA_i=Q\cap I_2\cap \cdots\cap I_u$.  Thus there exists $f_1,\ldots,f_l\in P$ and $a_1,\ldots,a_l\in (A_i)_Q$ such that
$a_1f_1+\cdots+a_lf_l=\hat z_1$. Write $a_i=\frac{b_i}{c_i}$ where $b_i\in A_i$ and $c_i\in A_i\setminus Q$.
Set $d_i=b_i\prod_{j\ne i}c_j$. Then
$$
d_1f_1+\cdots +d_lf_l=c\hat z_1
$$
where $c=\prod c_i\not\in Q$ (so $\overline \nu_i(c)<\infty$). 

We will show that there exists $g\in P$ such that $g$ has an expansion
\begin{equation}\label{eq22}
g=\alpha_1 \hat z_1+\alpha_2\hat z_2+\cdots+\alpha_nw_l
\mbox{ with }\alpha_i\in A_i\mbox{ and }\alpha_1\not\in Q.
\end{equation}
If one of the $f_i$ has such an expansion, then we set $g=f_i$. Otherwise,
$$
f_i\in (\hat z_1^2,\hat z_2,\ldots,w_l)
$$
for all $i$, and so, $c \hat z_1=\sum_{i=1}^ld_if_i\in (\hat z_1^2,\hat z_2,\ldots,w_l)$. But  $c\not\in Q$ implies 
$c\hat z_1\not\in (\hat z_1^2,\hat z_2,\ldots,w_l)$, a contradiction. Thus some $f_i$ has an expansion (\ref{eq22}).

Suppose $g\in A_i$ has an expansion (\ref{eq22}). 
Let $\tau=\overline \nu_i(\alpha_1)<\infty$. For $1\le j\le m$,
there exists $z_j\in U_i$ such that 
$\hat z_j=z_j+h_j$ where $h_j\in m_{A_i}^{3\tau}$. Substituting into (\ref{eq22}), we obtain an expansion
$$
 g=\delta_0+\delta_1  z_1+\delta_2 z_2+\cdots+\delta_nw_l
$$
with $\delta_0\in (m_{k_i[[x_1,\ldots,x_r]]})^{3\tau}$, $\delta_1,\ldots,\delta_n\in A_i$ and $\overline \nu_i(\delta_1)=\tau<\infty$.

Expand 
$$
\delta_1=\sum_{1\le i}\beta_ix_1^{b_{1,i}}\cdots x_r^{b_{r,i}}+\gamma_1 z_1+\cdots +\gamma_n w_l
$$
where $0\ne \beta_i\in k_i$, $\nu_i(x_1^{b_{1,i}}\cdots x_r^{b_{r,i}})<\nu_i(x_1^{b_{1,i+1}}\cdots x_r^{b_{r,i+1}})$
for all $i$, and $\gamma_1,\ldots,\gamma_n\in A_i$, so that $\overline\nu_i(\beta_ix_1^{b_{1,1}}\cdots x_r^{b_{r,1}})=\tau$.
 
 Expand 
$$
\delta_0=\sum_{1\le i}\epsilon_ix_1^{c_{1,i}}\cdots x_r^{c_{r,i}}
$$
where $0\ne \epsilon_i\in k_i$ and  $\nu_i(x_1^{c_{1,i}}\cdots x_r^{c_{r,i}})<\nu_i(x_1^{c_{1,i+1}}\cdots x_r^{c_{r,i+1}})$
for all $i$.  

Let $J=(\{x_1^{b_{1,i}}\cdots x_r^{b_{r,i}}\},\{ x_1^{c_{1,j}}\cdots x_r^{c_{r,j}}\})$ be the ideal generated by all of the 
$x_1^{b_{1,i}}\cdots x_r^{b_{r,i}}$ and $x_1^{c_{1,j}}\cdots x_r^{c_{r,j}}$. There exists a  transform $U_i\rightarrow U(1)$ along $\nu_i$ of type (1,i) in $x_1,\ldots,x_r$ and $x_1(1),\ldots,x_r(1)$ such that $JU(1)$ is generated by 
$x_1^{b_{1,1}}\cdots x_r^{b_{r,1}}=x_1(1)^{e_1}\cdots x_r(1)^{e_r}$.

Now define a sequence of  transforms $U(1)\rightarrow U(2)$ along $\nu_i$ of types   (4,i) and (2,i) by 
$$
z_1=x_1(1)^{e_1}\cdots x_r(1)^{e_r}z_1(1),
$$
$$
z_i=x_1(1)^{2e_1}\cdots x_r(1)^{2e_r}z_i(1)\mbox{ for } 2\le i\le m,
$$
$$
x_{j,k}=x_1(1)^{e_1}\cdots x_r(1)^{e_r}x_{j,k}(1)\mbox{ for }i+1\le j\le t\mbox{ and all }k,
$$
$$
w_j=x_1(1)^{2e_1}\cdots x_r(1)^{2e_r}w_j(1)\mbox{ for all $j$}.
$$

Taking the sequence of completions of these rings, 
$A_i\rightarrow A(1)=\widehat{U(1)}\rightarrow A(2)=\widehat{U(2)}$,
we have an induced  sequence of  formal transforms along $\overline\nu_i$. Then $g=x_1(1)^{e_1}\cdots x_r(1)^{e_r}g_1$ where $g_1\in m_{A(2)}$ satisfies $g_1\equiv \beta_1z_1(1)\mbox{ mod }m_{A(2)}^2$.
Now 
$$
\frac{g}{x_1(1)^{e_1}\cdots x_r(1)^{e_r}}\in \widehat{U(2)}\cap \mbox{QF}(U(2))=U(2)
$$
by \cite[Lemma 2]{Ab2}, and so $g_1\in U(2)\cap Q_{i+1}(A(2))=P_{i+1}(U(2))$. We thus have a reduction of $m$ in $A(2)$.

By induction on $m$, we have thus established the following theorem.

 \begin{Theorem}\label{Theorem2} Suppose that $U_i$ is as in (\ref{eq17}) and $U_i$ has a good regular system of parameters satisfying (\ref{eq21}). Then
 there exists a sequence of  transforms $U_i\rightarrow U(1)$ along $\nu_i$ of types (i,1), (i,2), (i,3) and  (i,4)such that $P_{i+1}(U(1))$ is a regular prime ideal in $U(1)$ and 
 $$
 Q_{i+1}(\widehat{U(1)})=P_{i+1}(U(1))\widehat{U(1)}.
 $$
 \end{Theorem}
 
 \subsection{Resolution  in arbitrary rank}
 
 Suppose that $U_i$ is as in (\ref{eq17}). Regular parameters 
 $$
 x_{i,1},\ldots,x_{i,r_i},y_1,\ldots,y_m
 $$
 in $U_i$ such that $\nu_i(x_{i,1}),\ldots,\nu_i(x_{i,r_i})$ is  a $\ZZ$-basis of $\Gamma_i/\Gamma_{i-1}$ and the prime ideal $P_{i+1}(U_i)$ of (\ref{eq18}) is 
$P_{i+1}(U_i)=(y_1,\ldots,y_m)$ so that $P_{i+1}(U_i)$ is a regular prime ideal are called $\nu_i$-good regular parameters in $U_i$. 
 
\begin{Lemma}\label{Lemma4} Suppose that $U_i$ is as in (\ref{eq17})  and that $U_i$ has $\nu_i$-good regular parameters 
$$
x_{i,1},\ldots,x_{i,r_i},y_1,\ldots,y_m.
$$
\begin{enumerate}
\item[1)] Suppose that $f\in U_i\setminus P_{i+1}(U_i)$. Then there exists a sequence of transforms $U_i\rightarrow U(1)$ along $\nu_i$ of types (1,i) and (2,i) such that $f$ has an expression
$$
f=x_{i,1}(1)^{d_1}\cdots x_{i,r_i}(1)^{d_{r_i}}\gamma,
$$
 where $d_1,\ldots,d_{r_i}\in\NN$ and $\gamma\in U(1)$ is a unit. 
\item[2)] Suppose that $f\in P_{i+1}(U_i)$ and $\rho\in \Gamma_i/\Gamma_{i-1}$ is given. Then there exists a sequence of transforms $U_i\rightarrow U(1)$ along $\nu_i$ of types (1,i) and (2,i) such that $f=x_{i,1}(1)^{d_1}\cdots x_{i,r_i}(1)^{d_{r_i}}\gamma$,  where $d_1,\ldots,d_{r_i}\in\NN$, $\gamma\in U(1)$ and 
$$
\nu_i(x_{i,1}(1)^{d_1}\cdots x_{i,r_i}(1)^{d_{r_i}})>\rho.
$$
\end{enumerate}
\end{Lemma}

The proof of Lemma \ref{Lemma4} is similar to the proof of Lemma \ref{Lemma1}.

\begin{proof} First assume that $f$ is in Case 1). Expand
$$
f=\sum \alpha_{\lambda_{i,1},\ldots,\lambda_{i,r_i}}x_{i,1}^{\lambda_{i,1}}\cdots x_{i,r_i}^{\lambda_{i,r_i}}
+h_1y_1+\cdots+h_my_m
$$
in $A_i=\hat U_i$ with $\alpha_{\lambda_{i,1},\ldots,\lambda_{i,r_i}}\in k_i$ not all zero and $h_1,\ldots h_m\in A_i$. 
Let 
$$
I=(x_{i,1}^{\lambda_{i,1}}\cdots x_{i,r_i}^{\lambda_{i,r_i}}\mid \alpha_{\lambda_{i,1},\ldots,\lambda_{i,r_i}}\ne 0),
$$
 an ideal in $U_i$.
By Proposition \ref{Perron}, there exists a sequence of transforms of type (1,i) along $\nu_i$, $U_i\rightarrow U(1)$, where $U(1)$ has $\nu_i$-good parameters $x_{i,1}(1),\ldots,x_{i,r_i}(1),y_1,\ldots,y_m$ and there exist $d_{i,1}\ldots,d_{i,r_i}\in \NN$, such that 
$$
IU(1)=x_{i,1}(1)^{d_{i,1}}\cdots x_{i,r_i}(1)^{d_{i,r_i}}U(1).
$$
Now perform a sequence of transforms of type (2,i) along $\nu_i$, $U(1)\rightarrow U(2)$ defined by
$$
y_j=x_{i,1}(1)^{d_{i,1}}\cdots x_{i,r_i}(1)^{d_{i,r_i}}y_j(1)\mbox{ for $1\le j\le m$,}
$$
 to obtain
$f=x_{i,1}(1)^{d_{i,1}}\cdots x_{i,r_i}(1)^{d_{i,r_i}}\gamma$ where $\gamma\in \widehat{U(2)}$ is a unit.
Now $\gamma\in \widehat{U(2)}\cap K$ implies $\gamma\in U(2)$, achieving the conclusions of 1) in $U(2)$.

Now suppose that $f$ is in Case 2). Then $f=h_1y_1+\cdots+h_my_m$ with $h_1,\ldots,h_m\in U_i$. Then perform a sequence of transforms along $\nu_i$ of type (2,i), $U_i\rightarrow U(1)$, defined by $y_j=x_{i,1}^{d_1}\cdots x_{i,r_i}^{d_{r_i}}y_j(1)$ for $1\le i\le m$, with $d_1,\ldots,d_{r_i}\in \NN$, such that $\nu_i(x_{i,1}^{d_1}\cdots x_{i,r_i}^{d_{r_i}})>\rho$, to obtain the conclusions of 2) in $U(1)$.

\end{proof}

\begin{Lemma}\label{Lemma2} Suppose that $U_i$ is as in (\ref{eq17})  and that $U_i$ has $\nu_i$-good regular parameters 
$$
x_{i,1},\ldots,x_{i,r_i},y_1,\ldots,y_m.
$$
Suppose that $z_1,\ldots,z_m$ are  regular parameters in $U_{i+1}=(U_i)_{P_{i+1}(U_i)}$. Then there exists a sequence of transforms along $\nu_i$, $U_i\rightarrow U(1)$   of types (1,i) and (2,i) such that $U_{i+1}=U(1)_{P_{i+1}(U(1))}$ and $U(1)$ has 
$\nu_i$-good regular parameters $x_{1,1}(1),\ldots,x_{1,r_1}(1), y_1(1),\ldots,y_m(1)$ such that $\nu_i(x_{i,1}(1)),\ldots,\nu_i(x_{1,r_1}(1))$ is a $\ZZ$-basis of $\Gamma_i/\Gamma_{i-1}$, $P_{i+1}(U(1))=(y_1(1),\ldots,y_m(1))$ is a regular prime and for $1\le j\le m$,
$$
z_j=x_{i,1}(1)^{d_j^1}\cdots x_{i,r_i}(1)^{d_j^{r_i}}y_j(1)
$$
with $d_k^j\in \NN$ for all $j,k$.
\end{Lemma}
 
\begin{proof} We have that 
$$
(y_1,\ldots,y_m)U_{i+1}=(z_1,\ldots,z_m)U_{i+1}=P_{i+1}(U_i)U_{i+1}.
$$
Thus there exist $a_{j,k}\in U_{i+1}$ such that for $1\le j\le m$,
$$
z_j=\sum_{k=1}^m a_{j,k}y_k.
$$
There exist $b_{j,k}\in U_i$ and $c_{j,k}\in U_i\setminus P_{i+1}(U_i)$ such that $a_{j,k}=\frac{b_{j,k}}{c_{j,k}}$.
 By Lemma \ref{Lemma4}, there exists a sequence of transforms $U_i\rightarrow U(1)$ along $\nu_i$ of types (1,i) and (2,i) such that $U(1)$ has $\nu_i$-good regular parameters
 $$
 x_{i,1}(1),\ldots,x_{i,r_i}(1),y_1(1),\ldots,y_m(1)
 $$
 such that for all $j,k$,
 $$
 c_{j,k}=x_{i,1}(1)^{d_{jk}^1}\cdots x_{i,r_i}(1)^{d_{jk}^{r_i}}\gamma_{jk}
 $$
 where $\gamma_{jk}$ are units in $U(1)$ and $d_{jk}^l\in \NN$. Now perform a sequence of transforms $U(1)\rightarrow U(2)$ along $\nu_i$ of type (2,i)
 $$
 y_k(1)=x_{i,1}(1)^{e_1^k}\cdots x_{i,r_i}(1)^{e_{r_i}^k}y_k(2)
 $$
 for $1\le k\le m$ so that we have expansions 
 $$
 z_j=\sum_{k=1}^m f_{jk}y_k(2)\mbox{ for }1\le j\le m
 $$
 with $f_{jk}\in U(2)$ for all $j,k$. We continue to have $U(2)_{P_{i+1}(U(2))}=U_{i+1}$. Since $\mbox{Det}(f_{jk})$ is a unit in $U_{i+1}$, there exists an $f_{jk}$ such that $f_{jk}\not\in P_{i+1}(U(2))$. Without loss of generality, $j=1$. By Lemma \ref{Lemma4}, there exists a sequence of transforms along $\nu_i$, $U(2)\rightarrow U(3)$, of types (1,i) and (2,i) such that $U(3)$ has  $\nu_i$-good regular  parameters 
 $x_{i,1}(3),\ldots,x_{i,r_i}(3),y_1(3),\ldots,y_m(3)$ such that
 $$
 f_{jk}=x_{i,1}(3)^{\alpha_{j,k}^1}\cdots x_{i,r_i}(3)^{\alpha_{j,k}^{r_i}}g_{j,k}
 $$
 where $g_{j,k}\in U(3)$ is a unit if $f_{jk}\not\in P_{i+1}(U(3))$ and $\nu_i(x_{i,1}(3)^{\alpha_{jk}^1}\cdots x_{i,r_i}(3)^{\alpha_{jk}^{r_i}})$ is arbitrarily large if $g_{jk}\in P_{i+1}(U(3))$. After performing a transform $U(3)\rightarrow U(4)$ along $\nu_i$ of type (i,1), and  permuting the $y_k(3)$, we have an expression
 $$
 z_1=x_{i,1}(4)^{\beta_1}\cdots x_{i,r_i}(4)^{\beta_{r_i}}\left[
 g_{1,1}y_1(4)+\sum_{k=2}^m g_{1,k}'y_k(4)\right]
 $$
 where $g_{jk}'\in U(4)$ and $g_{11}$ is a unit in $U(4)$. We then make a change of variables in $U(4)$, replacing $y_1(4)$ with $g_{1,1}y_1(4)+\sum_{k=2}^m g_{jk}'y_k(4)$, giving equations
 $$
 \begin{array}{lll}
 z_1&=&x_{i,1}(4)^{\beta_1}\cdots x_{i,r_i}(4)^{\beta_{r_i}} y_1(4)\\
 z_j&=& \sum_{k=1}^mg_{jk}'y_k(4)\mbox{ for $2\le j\le m$}
 \end{array}
 $$
  with $g_{jk}'\in U(4)$. We thus have that
  $$
  {\rm Det}\left(\begin{array}{ccc}
  g_{2,2}'&\cdots& g_{2,m}'\\
  \vdots&&\vdots\\
  g_{m,2}'&\cdots& g_{m,m}'  
  \end{array}\right)
  \not\in P_{i+1}(U(4)),
 $$
 so that some $g_{j,k}'\not\in P_{i+1}(U(4))$, with $2\le j, 2\le k$.
  We may thus continue as above to construct $U(4)\rightarrow U(5)$ such that $U(5)_{P_{i+1}(U(5))}=U_{i+1}$ and
 $$
 \begin{array}{lll}
 z_1&=&x_{i,1}(5)^{\beta_{1,1}}\cdots x_{i,r_i}(5)^{\beta_{1,r_i}} y_1(5)\\
 z_2&=&x_{i,1}(5)^{\beta_{2,1}}\cdots x_{i,r_i}(5)^{\beta_{2,r_i}} y_2(5)\\
 z_j&=& \sum_{k=1}^mg_{jk}'y_k(5)\mbox{ for $3\le j\le m.$}
 \end{array}
 $$ 
 By induction, we continue, to achieve the conclusions of the lemma.
 \end{proof}

 \begin{Lemma}\label{Lemma3} 
 Suppose that $U_i$ is as in (\ref{eq17})  and that $U_i$ has $\nu_i$-good regular parameters 
$$
x_{i,1},\ldots,x_{i,r_i},y_1,\ldots,y_m.
$$
Suppose that $U_{i+1}=(U_i)_{P_{i+1}(U_i)}$ has good regular parameters 
$$
\overline x_{i+1,1},\ldots,\overline x_{i+1},\ldots, \overline x_{t,1},\ldots,\overline x_{t,r_t}, \overline y_1,\ldots,\overline y_{\overline m}.
$$
Then there exist a sequence of transforms of types of types (1,i) and (2,i) $U_i\rightarrow U(1)$  along $\nu_i$   such that $U(1)$ has good regular parameters 
$$
x_{i,1}(1),\ldots,x_{i,r_i}(1),\ldots,x_{t,1}(1),\ldots,x_{t,r_t}(1),y_1(1),\ldots,y_{\overline m}(1)
$$
such that 
$$
\overline x_{a,b}=x_{i,1}(1)^{d_{a,b}^1}\cdots x_{i,r_i}(1)^{d_{a,b}^{r_i}}x_{a,b}(1)
$$
for $a\ge i+1$ and
$$
\overline y_l=x_{i,1}(1)^{e_l^1}\cdots x_{i,r_i}(1)^{e_l^{r_i}}y_l(1)
$$
for $1\le l\le m$. 

Suppose that
  $U_{i+1}\rightarrow X$ is a transform along $\nu_{i+1}$ of one of the types (1,j), (2,j), (3,j) or (4,j) with $i+1\le j$. Then there exists a sequence of transforms $U(1)\rightarrow V(1)$ along $\nu_i$ of types (1,k), (2,k), (3,k) and  (4,k) with $k=i$ or $k=j$ such that $V(1)_{P_{i+1}(V(1))}=X$.
\end{Lemma} 

\begin{proof}
The existence of the map $U_i\rightarrow U(1)$ having the properties asserted in the Lemma is a consequence of Lemma \ref{Lemma2}.

Suppose that $U_{i+1}\rightarrow X$ is of type (1,j). 
Then 
$X=U_{i+1}[\overline x_{j,1}(1),\ldots,\overline x_{j,r_j}(1)]_{\nu_{i+1}}$ where
$\overline x_{j,k}=\prod_{l=1}^{r_j}\overline x_{j,l}(1)^{a_{kl}}$ for $1\le k\le r_j$. 
We have that for $1\le l\le r_j$,
$$
\overline x_{j,l}(1)=\prod_{k=1}^{r_j}x_{k,l}(1)^{b_{k,l}}\delta_l
$$
where $\delta_l=\prod_{k=1}^{r_j}(x_{i,1}(1)^{d_{k,l}^1}\cdots x_{i,r_i}(1)^{d_{k,l}^{r_i}})^{b_{k,l}}$ is a unit in $U_{i+1}$. 

Now $\overline x_{j,l}(1)=\prod_{k=1}^{r_j}\overline x_{j,l}^{b_{k,l}}$ where $(b_{k,l})=(a_{k,l})^{-1}$ is a matrix with integral coefficients.  Thus $\nu_j(\prod_{k=1}^{r_k}x_{k,l}(1)^{b_{k,l}})>0$ for $1\le l\le r_j$.   
Defining $U(1)\rightarrow U(2)$ by $U(2)=U(1)[x_{j,1}(2),\ldots,x_{j,r_j}(2)]_{\nu_i}$ where $x_{j,k}(1)=\prod_{l=1}^{r_j} x_{j,l}(2)^{a_{k,l}}$ for $1\le k\le r_j$, we thus have that $U(2)_{P_{i+1}(U(2))}=X$.

Now suppose that $U_{i+1}\rightarrow X$ is of type (3,j). Then $X=U_{i+1}[\overline y_k(1)]_{\nu_{i+1}}$ where $\overline y_k=\overline x_{j,1}^{a_1}\cdots \overline x_{j,r_j}^{a_{r_j}}\overline y_k(1)$ with $\nu_j(\overline y_k(1))=0$ and $\nu_{i+1}(\overline y_k(1))\ge 0$. 
Now
$$
\nu_{i+1}\left(\frac{y_k(1)}{ x_{j,1}(1)^{a_1}\cdots  x_{j,r_j}(1)^{a_{r_j}}}\right)
=\nu_{i+1}\left(\frac{\overline y_k}{\overline x_{j,1}^{a_1}\cdots \overline x_{j,r_j}^{a_{r_j}}}\right)
\ge 0
$$
so if 
$$
\nu_{i}\left(\frac{y_k(1)}{x_{j,1}(1)^{a_1}\cdots x_{j,r_j}(1)^{a_{r_j}}}\right)<0,
\mbox{ then }
\nu_{i+1}\left(\frac{y_k(1)}{x_{j,1}(1)^{a_1}\cdots x_{j,r_j}(1)^{a_{r_j}}}\right)=0,
$$
and so there exists $k$ such that $a_k\ne 0$ and $n\in \NN$ such that 
$$
na_k\nu_i(x_{i,1}(1))+\nu_i \left(\frac{y_k(1)}{x_{j,1}(1)^{a_1}\cdots x_{j,r_j}(1)^{a_{r_j}}}\right)
\ge 0.
$$
Let $U(1)\rightarrow U(2)$ be the transform of type (2,i) defined by $x_{j,k}(1)=x_{i,1}(1)^n x_{j,k}(2)$. Then, setting $x_{j,l}(2)=x_{j,k}(1)$ for $l\ne k$, define $U(2)\rightarrow U(3)$ to be the transform of type (3,j) defined by $ y_k(1)=x_{j,1}(2)^{a_1}x_{j,2}(2)^{a_2}\cdots x_{j,r_j}(1)^{a_{r_j}}y_k(2)$.

The remaining two cases, transforms of types (2,j) and (4,j), have a similar but simpler analysis.
\end{proof}

\begin{Theorem}\label{Theorem3} Let $D$ be the local ring of (\ref{eq24}). There exists a sequence of transforms $D\rightarrow D(1)$ along $\nu$ such that 
there exist  good  regular parameters 
$$
x_{1,1},\ldots,x_{t,r_t},y_{1},\ldots,y_m
$$
 in $D(1)$ such that 
$\nu_j(x_{j,1}),\ldots,\nu_j(x_{j,r_j})$ is a $\ZZ$-basis of $\Gamma_j/\Gamma_{j-1}$ for $1\le j\le t$ and 
$$
\nu(y_{1})=\cdots = \nu(y_m)=\infty.
$$
 In particular, $P_j(D(1))$ are regular primes in $D(1)$ for $1\le j\le t+1$ and thus $D(1)/P_{t+1}(D(1))$ is a regular local ring which is dominated by $\nu$ and dominates $B_{\nu_1}$.
We further have that 
$$
\left(D(1)/P_i(D(1))\right)_{P_i(D(1))}\cong V_{\nu_i}/m_{\nu_i}
$$
for $1\le i\le t$.
\end{Theorem}

\begin{proof} We will prove the theorem by descending induction on $i$ with $1\le i\le t$. By Theorems \ref{Theorem1} and  \ref{Theorem2}, there exists a sequence of transforms $D_t\rightarrow E_1$ along $\nu_t$ of types (1,t), (2,t), (3,t) and (4,t) such that $P_{t+1}(E_1)$ is a regular prime in $E_1$. 

Suppose, by induction, that we have constructed a sequence of transforms $D_{i+1}\rightarrow E_1$ along $\nu_{i+1}$ of types (1,j), (2,j), (3,j) and (4,j) with $j\ge i+1$ such that $P_{j}(E_1)$ are regular primes in $E_1$ for $j\ge i+1$. By Theorems \ref{Theorem1} and  \ref{Theorem2} and Lemmas \ref{Lemma2} and \ref{Lemma3}, there exists a sequence of transform $D_i\rightarrow F_1$ along $\nu_i$ of types (1,j), (2,j), (3,j) and (4,j) with $j\ge i$ such that $(F_1)_{P_{i+1}(F_1)}=E_1$. By Theorem \ref{Theorem2}, there exists a sequence of transforms along $\nu_i$, $F_1\rightarrow F_2$, such that $(F_2)_{P_{i+1}(F_2)}=E_1$ and $P_{i+1}(F_2)$ is a regular prime in $F_2$. By Lemmas \ref{Lemma2} and \ref{Lemma3}, there exists a sequence of transforms along $\nu_i$, $F_2\rightarrow F_3$, such that $F_3$ has good regular parameters
$x_{i,1},\ldots,x_{t,r_t},z_1,\ldots,z_m$ such that $P_j(F_3)=(x_{j,1},\ldots,y_m)$ for all $j\ge i$ and 
$\nu_i(x_{j,1}),\ldots,\nu_i(x_{j,t_j})$ is a $\ZZ$-basis of $\Gamma_j/\Gamma_{j-1}$ for $i\le j\le t$.

The last statement of the theorem follows from (\ref{eq23}).
\end{proof}

\section{Local Uniformization of Abhyankar valuations}\label{SecLocUnif}

In this section we prove Theorems \ref{TheoremA}, \ref{TheoremB}, \ref{TheoremC} and \ref{TeoAbh}.
\vskip .2truein
\noindent{\bf Proof of Theorem \ref{TheoremA}:} This is immediate from Theorem \ref{Theorem3}, taking $R$ to be $$
D(1)/P_{t+1}(D(1)).
$$
\vskip .2truein
We remark that 
the regular parameters in $R$ of the conclusions of Theorem \ref{TheoremB} are good regular parameters (Definition \ref{Def1}).
\vskip .2truein
\noindent{\bf Proof of Theorem \ref{TheoremB}:} We first prove 1). In $\hat R=k_1[[x_{1,1},\ldots,x_{t,r_t}]]$, where $k_1\cong V_{\nu}/m_{\nu}$ is a coefficient field of $\hat R$, we have an expansion

\begin{equation}\label{eq26}
f=\sum \alpha_{b_{1,1},\ldots,b_{t,r_t}} x_{1,1}^{b_{1,1}}\cdots x_{t,r_t}^{b_{t,r_t}}
\end{equation}
with $\alpha_{b_{1,1},\ldots,b_{t,r_t}}\in k_1$. Let $J$ be the ideal
$$
J=(x_{1,1}^{b_{1,1}}\cdots x_{t,r_t}^{b_{t,r_t}}\mid \alpha_{b_{1,1},\ldots,b_{t,r_t}}\ne 0).
$$
By Proposition \ref{Perron}, there exists a sequence of primitive transforms $R\rightarrow R(1)$ along $\nu$ such that 
$$
JR(1)=x_{1,1}(1)^{a_{1,1}}\cdots x_{t,r_t}(1)^{a_{t,r_t}}R(1).
$$
Then $f$ has an expression 
$$
f=x_{1,1}(1)^{a_{1,1}}\cdots x_{t,r_t}(1)^{a_{t,r_t}}u
$$
with $u\in\widehat{R(1)}$ a unit. By \cite[Lemma 2]{Ab2}, $u\in K\cap \widehat{R(1)}=R(1)$, giving the desired expression of $f$ in $R(1)$.

To prove 2), take generators $f_1,\ldots,f_m$ of $I$. By part 1) of this theorem, there exists a sequence of primitive transforms $R\rightarrow R(1)$ along $\nu$ such that each $f_i$ is a monomial in $x_{1,1}(1),\ldots, x_{t,r_t}(1)$ times a unit in $R(1)$. By Proposition \ref{Perron}, we many now apply another sequence of primitive transforms $R(1)\rightarrow R(2)$ along $\nu$ to achieve the conclusions of 2).

The proof of 3) is a variation on the proof of 1), as in Lemma \ref{Lemma1}.
\vskip .2truein
\noindent{\bf Proof of Theorem \ref{TheoremC}:} There exist $f_1,\ldots,f_m\in V_{\nu}$ such that $S=k[f_1,\ldots,f_m]_{\nu}$. Let $R$ be the regular local ring of the conclusions of Theorem \ref{TheoremA}. By Theorem \ref{TheoremB}, there exists a sequence of primitive transforms $R\rightarrow R(1)$ such that $f_1,\ldots,f_m\in R(1)$ and $R(1)$ satisfies the conclusions of Theorem \ref{TheoremA}. Thus $R(1)$ dominates $S$ and so $R(1)$ satisfies the conclusions of Theorem \ref{TheoremC}.

\vskip .2truein
\noindent{\bf Proof of Theorem \ref{TeoAbh}:}  To prove Theorem \ref{TeoAbh}, we need only modify  the proof of \cite[Theorem 1.5]{CN} by observing that
the statement of  \cite[Theorem 7.2]{CN} is true without the  assumption that  $V_{\nu}/m_{\nu}$ is separable over $k$, using Theorems \ref{TheoremA} and \ref{TheoremB} of this paper in place of \cite[Theorem 1.1]{KK}.

 \end{document}